\documentclass[10pt, onecolumn, draftclsnofoot]{IEEEtran}

\usepackage{amsthm}
\usepackage{amsbsy}
\usepackage{amsmath}
\usepackage{amssymb}
\usepackage{latexsym}
\usepackage{ifthen}
\usepackage{subfigure}
\usepackage{turnstile}
\usepackage{mathtools}
\usepackage{booktabs}
\usepackage{balance}
\usepackage{paralist}
\usepackage{cite}
\usepackage{url}

\newtheorem{remark}{Remark}[section]

\newtheorem{assumption}{Assumption}[section]
\newtheorem{definition}{Definition}[section]
\newtheorem{proposition}{Proposition}[section]
\newtheorem{theorem}{Theorem}[section]
\newtheorem{lemma}{Lemma}[section]
\newtheorem{corollary}{Corollary}[section]

\newcommand{\vecc}{\boldsymbol{\operatorname{Vec}}}
\newcommand{\ndiag}{\boldsymbol{\operatorname{Diag}}}
\newcommand{\argmax}{\boldsymbol{\operatorname{argmax}}}

\def \Vap{\varepsilon}

\def \C{\mathcal{C}}
\def \PC{\mathcal{C}^{\perp}}
\def \OL{\overline{L}}

\def \bmu{\boldsymbol{\mu}}
\def \bnu{\boldsymbol{\nu}}
\def \btheta{\boldsymbol{\theta}}

\def \wz{\widehat{\mathbf{z}}}
\def \abtheta{\btheta^{\ast}}
\def \abbtheta{\overline{\btheta}^{\ast}}
\def \bh{\overline{h}}
\def \bK{\overline{K}}
\def \bg{\overline{g}}
\def \bc{\overline{b}}
\def \mz{\mathbf{z}}
\def \bz{\breve{\mathbf{z}}}
\def \bU{\overline{U}}
\def \bJ{\overline{J}}
\def \pz{\mathbf{z}^{\prime}}
\def \omu{\overline{\mu}}
\def \bx{\breve{\mathbf{x}}}
\def \wtheta{\widehat{\btheta}}
\def \tri{t\rightarrow\infty}
\def \bG{\breve{G}}
\def \bI{\breve{I}}
\def \bV{\breve{V}}
\def \wid{\widehat{d}}
\def \bzeta{\boldsymbol{\zeta}}
\def \bv{\breve{\mathbf{v}}}
\def \bu{\breve{\mathbf{u}}}
\def \mpu{\mathbf{U}^{\prime}}
\def \wu{\widehat{\mathbf{u}}}
\def \wU{\widehat{\mathbf{U}}}
\def \bau{\bu^{a}}
\def \oc{\overline{c}}
\def \mR{\mathcal{R}}
\def \wmR{\widehat{\mR}}
\def \blambda{\overline{\lambda}}
\def \Past{\mathbb{P}_{\abtheta}}
\def \East{\mathbb{E}_{\abtheta}}

\mathtoolsset{showonlyrefs}

\title{Asymptotically Efficient Distributed Estimation With Exponential Family Statistics}

\author{Soummya Kar and Jos\'e M.~F.~Moura

\thanks{The authors are with the Department of Electrical and Computer Engineering,
Carnegie Mellon University, Pittsburgh, PA 15213, USA (soummyak@andrew.cmu.edu, moura@ece.cmu.edu).}

\thanks{The work was partially supported by NSF grants \#~CCF-1018509 and~CCF-1011903 and by AFOSR grant \#~FA9550101291.}}

\begin{document}
\maketitle

\renewcommand{\thefootnote}{\fnsymbol{footnote}}

\renewcommand{\thefootnote}{\arabic{footnote}}

\begin{abstract}
The paper studies the problem of distributed parameter estimation in multi-agent networks with exponential family observation statistics. A certainty-equivalence type distributed estimator of the \emph{consensus} + \emph{innovations} form is proposed in which, at each each observation sampling epoch agents update their local parameter estimates by appropriately combining the data received from their neighbors and the locally sensed new information (innovation). Under global observability of the networked sensing model, i.e., the ability to distinguish between different instances of the parameter value based on the joint observation statistics, and mean connectivity of the inter-agent communication network, the proposed estimator is shown to yield consistent parameter estimates at each network agent. Further, it is shown that the distributed estimator is asymptotically efficient, in that, the asymptotic covariances of the agent estimates coincide with that of the optimal centralized estimator, i.e., the inverse of the centralized Fisher information rate. From a technical viewpoint, the proposed distributed estimator leads to non-Markovian mixed time-scale stochastic recursions and the analytical methods developed in the paper contribute to the general theory of distributed stochastic approximation.
\end{abstract}

\begin{IEEEkeywords} Distributed estimation, exponential family, consistency, asymptotic efficiency, stochastic approximation.\end{IEEEkeywords}



\section{Introduction}
\label{sec:introduction}

\subsection{Motivation}
\label{subsec:mot} Motivated by applications in multi-agent networked information processing, we revisit the problem of distributed sequential parameter estimation. The setup considered is a highly non-classical distributed information setting, in which each network agent samples over time an independent and identically distributed (i.i.d.) time-series with exponential family statistics\footnote{Exponential families subsume most of the distributions encountered in practice, for example, Gaussian, gamma, beta etc.} parameterized by the (vector) parameter of interest. The observation sequences are assumed to be conditionally independent (conditioned on the true parameter value) across the agents with different statistics. Further, in the spirit of typical agent-networking and wireless sensing applications with limited agent communication and computation capabilities, we restrict ourselves to scenarios in which each agent is only aware of its local observation statistics and, assuming slotted-discrete time, may only communicate (collaborate) with its agent-neighborhood (possibly dynamic and random) once per epoch of new observation acquisition, i.e., we consider scenarios in which the inter-agent communication rate is at most as high as the observation sampling rate. Broadly speaking, the goal of distributed parameter estimation in such multi-agent scenarios is to update over time the local agent estimates by effectively processing local observation samples and exchanging information with neighboring agents. To this end, the paper presents a distributed estimation approach of the consensus + innovations type, which, among other things, accomplishes the following:

\noindent {\bf Consistency under distributed observability}: Under \emph{global observability}\footnote{Global observability means that for every pair of different parameter values, the corresponding probability measures induced on the aggregate or collective agent observation set are \emph{distinguishable}. For setups involving exponential families distinguishability is aptly captured by strict positivity of the Kullback-Liebler (KL) divergence between the corresponding measures, see Assumption~\ref{ass:globobs} for details.} of the multi-agent sensing model and \emph{mean connectivity} of the inter-agent communication-collaboration network, our distributed estimation approach is shown to yield strongly consistent parameter estimates at each agent. Conversely, it may be readily seen that the conditions of global observability and mean network connectivity are in fact necessary for obtaining consistent parameter estimates in our distributed information-collaboration setup. Indeed, global observability is the minimal requirement for consistency even in centralized estimation, whereas, in the absence of network connectivity, there may be locally unobservable agent-network components which, under no circumstance, will be able to generate consistent parameter estimates.


\noindent {\bf Asymptotic efficiency}: Under the same conditions of global observability of the multi-agent sensing model and mean connectivity of the inter-agent communication-collaboration network, the proposed distributed estimation approach is shown to be asymptotically efficient. In other words, in terms of asymptotic convergence rate, the local agent estimates are as good as the optimal centralized\footnote{The term centralized estimator refers to a hypothetical fusion center based estimator that has access to all agent observations at all times.}, i.e., the local estimates achieve asymptotic covariance equal to the inverse of the centralized Fisher information rate. The key point to note here is that the above optimality holds as long as the mean communication network is connected irrespective of how sparse the link realizations are.

In the context of parallel computing and optimization in multi-agent environments, interacting stochastic gradient and stochastic approximation algorithms have been proposed--see, for example, early work~\cite{tsitsiklisphd84,tsitsiklisbertsekasathans86,Bertsekas-survey,Kushner-dist}. In contrast, to cope with scenarios where local observations are sensed sequentially over time and inter-agent communication is restricted to arbitrary preassigned topologies and occurs at the same rate as sensing, we have proposed \emph{consensus} + \emph{innovations} type architectures, see~\cite{KarMouraRamanan-Est-2008}. \emph{Consensus} + \emph{innovations} algorithms embed a single round of neighborhood consensus or agreement like in~\cite{Bertsekas-survey,olfatisaberfaxmurray07,Dimakis-Gossip-SPM-2011,jadbabailinmorse03}, with in addition local processing of the sampled new observation, the local innovation; see for example
consensus + innovation approaches for nonlinear distributed estimation~\cite{KarMouraRamanan-Est-2008}, detection~\cite{bajovicjakoveticxaviresinopolimoura-11,jakovetic2012distributed,Kar-Tandon-ISIT-2011}, adaptive control~\cite{Kar-Bandit-CDC-2011} and learning~\cite{Kar-QD-learning-TSP-2012}. Other approaches for distributed optimization and inference in multi-agent networks have been considered, see for example diffusion for network inference and optimization~\cite{Sayed-LMS,chen2012diffusion} and networked LMS and variants~\cite{Sayed-LMS,Stankovic-parameter,Giannakis-LMS,Nedic-parameter,sundhar2010distributed}. The key distinction between this prior art and the current paper is that, in the former the focus has been mainly on consistency (or minimizing the asymptotic error residual between the estimated and the true parameter), but not on asymptotic efficiency. The requirement of asymptotic efficiency complicates the construction of such distributed algorithms non-trivially and necessitates the use of time-varying consensus and innovation gains in the update process; further these time-varying gains driving the persistent consensus and innovation potentials need to decay at strictly different rates in order for the distributed scheme to achieve the asymptotic covariance of the optimal centralized estimator. Such mixed time-scale construction for asymptotically efficient distributed parameter estimation in linear statistical models was obtained in~\cite{KarMoura-LinEst-JSTSP-2011,Kar-AdaptiveDistEst-SICON-2012}. However, in contrast to optimal estimation in linear statistical models~\cite{KarMoura-LinEst-JSTSP-2011,Kar-AdaptiveDistEst-SICON-2012}, in the nonlinear non-Gaussian setting, the local innovation gains that achieve asymptotic efficiency are necessarily dependent on the true value of the parameter to be estimated and on the statistics of the global sensing model. Since the value of the parameter (and hence the optimal estimator gains) are not available in advance, our proposed distributed estimation approach involves a distributed online gain learning procedure that proceeds in conjunction with the sequential estimation task. As a result, a closed-loop interaction occurs between the gain learning and parameter estimation that is reminiscent of the certainty-equivalence approach for adaptive estimation and control--although the analysis methodology is significantly different from classical techniques used in adaptive processing (see, for example,~\cite{Lai-Wei,Lai-nonlinlsad} and the references therein, in the context of parameter estimation), primarily due to the distributed nature of our problem. Specifically, in our approach, each agent runs simultaneously three local time recursions: (1) an auxiliary distributed consensus + innovations estimator driven by
non-adaptive innovation gains; (2) an online distributed learning procedure that uses
the auxiliary distributed local estimators to generate a sequence of optimal adaptive
innovations gains; and 3) the desired distributed consensus + innovations estimator
whose innovations are weighted by the optimal adaptive innovations gains, thus achieving asymptotic efficiency. We note in this context that the idea of recovering asymptotically efficient estimates from consistent (but suboptimal) auxiliary estimates, although novel from a distributed estimation standpoint, has been investigated in prior work on (centralized) recursive estimation, see, for example,~\cite{Hasminskii1974estimation,Fabian1978efficient}. Finally, we note that the current formulation assumes unconstrained parametrization, in that, the parameter may take values in all of $\mathbb{R}^{M}$. There exist some related work on distributed stochastic optimization (see, for example,~\cite{ram2010distributed}) where distributed iterative algorithms that include a (local) projection step have been proposed to ensure convergence or consistency. While he do not consider constrained parameter estimation in this paper, possible extensions of our approach to constrained formulations are discussed in Section~\ref{sec:conclusions}.

In summary, in contrast to existing work, the current paper presents a principled development of distributed parameter estimation as applicable to the general and important class of multi-agent statistical exponential families; paralleling the classical development of centralized parameter estimation, it quantifies notions of (distributed) observability, performance metrics, information measures, and algorithmic optimality. Due to the mixed time-scale behavior and the non-Markovianity (induced by the learning process), the stochastic procedure does not fall under the purview of standard stochastic approximation (see, for example,~\cite{Nevelson}) or distributed stochastic approximation (see, for example,~\cite{tsitsiklisbertsekasathans86,Bertsekas-survey,Kushner-dist,KarMouraRamanan-Est-2008,Stankovic-parameter,Li-Feng,Huang,KarMoura-DistCons-TSP-2009,sundhar2010distributed}) procedures. In fact, some of the intermediate results on the pathwise convergence rates of mixed time-scale stochastic procedures obtained in the paper are more broadly applicable and contribute to the general theory of distributed stochastic approximation. In this context, we note the study of mixed time-scale stochastic procedures that arise in algorithms of the simulated annealing type (see, for example,~\cite{Gelfand-Mitter}). Apart from being distributed, our scheme technically differs from~\cite{Gelfand-Mitter} in that, whereas the additive perturbation in~\cite{Gelfand-Mitter} is a martingale difference sequence, ours is a network dependent consensus potential manifesting past dependence. In fact, intuitively, a key step in the analysis is to derive pathwise strong approximation results to characterize the rate at which the consensus term/process converges to a martingale difference process. We also emphasize that our notion of mixed time-scale is different from that of stochastic algorithms with coupling (see~\cite{Borkar-stochapp,Yin-book}), where a quickly switching parameter influences the relatively slower dynamics of another state, leading to \emph{averaged} dynamics. Mixed time scale procedures of this latter type arise in multi-scale distributed information diffusion problems, see, in particular, the paper~\cite{Krishnamurthy-Yin-consensus}, that studies interactive consensus formations in Markov-modulated switching networks.

\noindent {\bf A detailed look into existing consensus + innovations and diffusion approaches}: We discuss two broad approaches for dynamic information processing in distributed multi-agent networks (i.e., in which agents sense and process new data and cooperate or interact with each other simultaneously), namely the \emph{diffusion} approaches~\cite{Sayed-LMS,chen2012diffusion} and the \emph{consensus} + \emph{innovations} approaches~\cite{KarMouraRamanan-Est-2008,KarMoura-LinEst-JSTSP-2011,Kar-Moura-SPM-2013}, that are closest to the current line of work. While both develop distributed stochastic recursive schemes that update the local decision variables (estimates in this case) by combining the data received from the neighbors and the new external information sensed in the same time step, a key difference is in the innovation gain selection processes--in the diffusion algorithms the innovation gains (the new information fusion weights) are taken to be constant, whereas, in the consensus + innovation schemes these weights are adaptive and made to decay over time in a controlled fashion. These choices have important consequences as far as the qualitative convergence behavior is concerned: the constant gains in the diffusion approaches, while may facilitate adaptation in dynamic parameter environments, pay the price of non-zero steady state errors or inconsistent estimates (see, for example, \cite{Sayed-LMS} in the context of quadratic stochastic optimization and linear parameter estimation), whereas, the adaptive innovation weight selection guarantees consistent estimation (with zero steady state errors) in the consensus + innovation approaches (see, for example,~\cite{KarMouraRamanan-Est-2008,KarMoura-LinEst-JSTSP-2011} in the context of linear and nonlinear parameter estimation). Moreover, in~\cite{KarMouraRamanan-Est-2008,KarMoura-LinEst-JSTSP-2011} it was shown that the agent estimates are asymptotically normal, i.e., at each agent $n$ the estimation error covariance $V_{n}(t)$ goes down as $O(1/t)$, and the corresponding asymptotic covariance $\bar{V}_{n}$, i.e., the limit as $t\rightarrow\infty$ of the scaled quantity $tV_{n}(t)$, was characterized explicitly as a function of the network and sensing models (Note that, as far as diffusion schemes~\cite{Sayed-LMS,chen2012diffusion} are concerned, the quantity $tV_{n}(t)$ blows up as $t\rightarrow\infty$ since the corresponding error covariance $V_{n}(t)$ stays bounded away from zero as mentioned earlier.) However, although the consensus + innovation approaches in~\cite{KarMouraRamanan-Est-2008} were shown to be asymptotically normal, they were not asymptotically efficient in general, i.e., the agent asymptotic covariances $\bar{V}_{n}$'s were strictly larger than the inverse of the centralized Fisher information rate, the asymptotic covariance attainable by an optimal centralized procedure. Later in~\cite{KarMoura-LinEst-JSTSP-2011} asymptotically efficient distributed estimation procedures for linear models were achieved by a more delicate adaptive tuning of both the consensus and innovation gain sequences. In this paper, we develop asymptotically efficient procedures for a much larger class of nonlinear models.

We comment briefly on the organization of the rest of the paper. Section~\ref{notgraph} sets up notation. The multi-agent sensing model is formalized in Section~\ref{subsec:sensmod}, whereas, preliminary facts pertaining to the model and assumptions are summarized in Section~\ref{subsec:prel-exp}. Section~\ref{subsec:alg} describes the distributed estimation approach and the main results of the paper (concerning consistency and asymptotic efficiency of the proposed approach) are stated in Section~\ref{subsec:mainres}. The major technical developments are presented in Section~\ref{sec:genconsest} culminating to the proofs of the main results in Section~\ref{sec:proof_main_res}. A detailed discussion on the implications of the major technical constructs, comparisons with existing approaches, complexity of implementation and some trade-offs is provided in Section~\ref{sec:disc}. Finally, Section~\ref{sec:conclusions} concludes the paper.

\subsection{Notation}
\label{notgraph} We denote by~$\mathbb{R}$ the set of reals, $\mathbb{R}_{+}$  the set of non-negative reals,  and by~$\mathbb{R}^{k}$ the $k$-dimensional Euclidean.  For $a,b\in\mathbb{R}$, we use $a\vee b$ and $a\wedge b$ to denote the maximum and minimum of $a$ and $b$ respectively. For deterministic $\mathbb{R}_{+}$-valued sequences $\{a_{t}\}$ and $\{b_{t}\}$, the notation $a_{t}=O(b_{t})$ denotes the existence of a constant $c>0$ such that $a_{t}\leq cb_{t}$ for all $t$ sufficiently large. Further, the notation $a_{t}=o(b_{t})$ is used to indicate that $a_{t}/b_{t}\rightarrow 0$ as $\tri$. For $\mathbb{R}_{+}$-valued stochastic processes $\{a_{t}\}$ and $\{b_{t}\}$, the corresponding order notations are to be interpreted to hold pathwise almost surely (a.s.).

The set of $k\times k$ real matrices is denoted by $\mathbb{R}^{k\times k}$. The corresponding subspace of symmetric matrices is denoted by $\mathbb{S}^{k}$. The cone of positive semidefinite matrices is denoted by $\mathbb{S}_{+}^{k}$, whereas $\mathbb{S}_{++}^{k}$ denotes the subset of positive definite matrices. The $k\times k$ identity matrix is
denoted by $I_{k}$, while $\mathbf{1}_{k}$ and $\mathbf{0}_{k}$ denote
respectively the column vector of ones and zeros in
$\mathbb{R}^{k}$. Often the symbol $\mathbf{0}$ is used to denote the $k\times p$ zero matrix, the dimensions being clear from the context. The symbol $\top$ denotes matrix transpose, whereas, for a finite set of matrices $A_{n}\in\mathbb{R}^{k_{n}\times p}$, $n=1,\cdots,N$, the quantity $\vecc(A_{n})$ denotes the $(k_{1}+\cdots+k_{N})\times p$ matrix $[A_{1}^{\top},\cdots,A_{N}^{\top}]^{\top}$ obtained as the (column-wise) stack of the matrices $A_{n}$. The operator
$\left\|\cdot\right\|$ applied to a vector denotes the standard
Euclidean $\mathcal{L}_{2}$ norm, while applied to matrices it denotes the induced
$\mathcal{L}_{2}$ norm, which is equivalent to the matrix spectral radius for symmetric
matrices. Also, for $\mathbf{a}\in\mathbb{R}^{k}$ and $\Vap>0$, we will use $\mathbb{B}_{\Vap}(\mathbf{a})$ to denote the closed $\Vap$-neighborhood of $\mathbf{a}$, i.e.,
\begin{equation}
\label{notgraph1}
\mathbb{B}_{\Vap}(\mathbf{a})=\left\{\mathbf{b}\in\mathbb{R}^{k}~:~\|\mathbf{b}-\mathbf{a}\|\leq\Vap\right\}.
\end{equation}
The notation $A\otimes B$ is used to denote the Kronecker product of two matrices $A$ and $B$.

The following notion of consensus subspace and its complement will be used:
\begin{definition}
\label{def:consspace} Let $N$ and $M$ be positive integers and consider the Euclidean space $\mathbb{R}^{NM}$. The consensus or agreement subspace $\mathcal{C}$ of $\mathbb{R}^{NM}$ is then defined as
\begin{equation}
\label{def:consspace1}\mathcal{C}=\left\{\mathbf{z}\in\mathbb{R}^{NM}~:~\mathbf{z}=\mathbf{1}_{N}\otimes\mathbf{a}~\mbox{for some $\mathbf{a}\in\mathbb{R}^{M}$}\right\}.
\end{equation}
The orthogonal complement of $\mathcal{C}$ in $\mathbb{R}^{NM}$ is denoted by $\PC$. Finally, for a given vector $\mathbf{z}\in\mathbb{R}^{NM}$, its projection on the consensus subspace $\C$ is to be denoted by $\mathbf{z}_{\C}$, whereas, $\mathbf{z}_{\PC}=\mathbf{z}-\mathbf{z}_{\C}$ denotes the projection on the orthogonal complement~$\PC$.

Also, for $\mathbf{z}\in\mathcal{C}$, we will denote by $\mathbf{z}^{a}$ the vector $\mathbf{a}\in\mathbb{R}^{M}$ such that $\mathbf{z}=\mathbf{1}_{N}\otimes\mathbf{a}$.
\end{definition}

Time is assumed to be discrete or slotted throughout the paper. The symbols $t$ and $s$  denote time, and $\mathbb{T}_{+}$ is the discrete index set $\{0,1,2,\cdots\}$. The parameter to be estimated belongs to a subset~$\Theta$ (generally open) of the Euclidean space $\mathbb{R}^{M}$. We reserve the symbol $\btheta$ to denote a canonical element of the parameter space $\Theta$, whereas, the true (but unknown) value of the
parameter (to be estimated) is denoted by~$\abtheta$. The symbol $\mathbf{x}_{n}(t)$ is used to denote the $\mathbb{R}^{M}$-valued
estimate of~$\abtheta$ at time~$t$ at agent~$n$, whereas, $\mathbf{y}_{n}(t)$ will be used to denote the observation at agent $n$ at time $t$. Without
loss of generality, the initial estimate,
$\mathbf{x}_{n}(0)$, at time~$0$ at agent~$n$ is assumed to be a non-random
quantity.

\textbf{Spectral graph theory}: The inter-agent communication topology at a given time instant may be described by an \emph{undirected} graph $G=(V,E)$, with $V=\left[1\cdots N\right]$ and~$E$ denoting the set of agents (nodes) and inter-agent communication links (edges) respectively. The unordered pair $(n,l)\in E$ if there exists an edge between nodes~$n$ and~$l$. We consider simple graphs, i.e., graphs devoid of self-loops and multiple edges. A graph is connected if there exists a path\footnote{A path between nodes $n$ and $l$ of length $m$ is a sequence
$(n=i_{0},i_{1},\cdots,i_{m}=l)$ of vertices, such that $(i_{k},i_{k+1})\in E\:\forall~0\leq k\leq m-1$.}, between each pair of nodes. The neighborhood of node~$n$ is
\begin{equation}
\label{def:omega} \Omega_{n}=\left\{l\in V\,|\,(n,l)\in
E\right\}. 
\end{equation}
Node~$n$ has degree $d_{n}=|\Omega_{n}|$ (the number of edges with~$n$ as one end point.) The structure of the graph can be described by the symmetric $N\times N$ adjacency matrix, $A=\left[A_{nl}\right]$, $A_{nl}=1$, if $(n,l)\in E$, $A_{nl}=0$, otherwise. Let the degree matrix  be the diagonal matrix $D=\mbox{diag}\left(d_{1}\cdots d_{N}\right)$. By definition, the positive semidefinite matrix $L=D-A$ is called the graph Laplacian matrix. The eigenvalues of $L$ can be ordered as $0=\lambda_{1}(L)\leq\lambda_{2}(L)\leq\cdots\leq\lambda_{N}(L)$, the eigenvector corresponding to $\lambda_{1}(L)$ being $(1/\sqrt{N})\mathbf{1}_{N}$. The multiplicity of the zero eigenvalue equals the number of connected components of the network; for a connected graph, $\lambda_{2}(L)>0$. This second eigenvalue is the algebraic connectivity or the Fiedler value of
the network; see \cite{FanChung} for detailed treatment of graphs and their spectral theory.

\section{Multi-agent sensing model}
\label{sec:sensmod} Let $\btheta^{\ast}\in\mathbb{R}^{M}$ be an $M$-dimensional (vector) parameter that is to be estimated by a network of~$N$ agents. Throughout, we assume that all the random objects are defined on a common measurable space $\left(\Omega,\mathcal{F}\right)$ equipped with a filtration $\{\mathcal{F}_{t}\}$. Probability and expectation, when the true (but unknown) parameter value $\abtheta$ is in force, are denoted by $\Past(\cdot)$ and  $\East[\cdot]$ respectively. All inequalities involving random variables are to be interpreted a.s.

Since the \emph{sources of randomness} in our formulation are the observations $\mathbf{y}_{n}(t)$'s sensed by the network agents at each time $t$ and the Laplacian matrices $L_{t}$'s modeling the stochastic inter-agent communication graphs over time (to be made precise soon), the filtration $\{\mathcal{F}_{t}\}$ may be taken to be the natural filtration induced by these random quantities, i.e., $\mathcal{F}_{t}=\sigma\left(\{L_{s},\{\mathbf{y}_{n}(s)\}_{n=1}^{N}\}_{s=0}^{t-1}\right)$ is the $\sigma$-algebra induced by the observation and communication processes. Finally, a stochastic process $\{\mathbf{z}_{t}\}$ is said to be $\{\mathcal{F}_{t}\}$-adapted if the $\sigma$-algebra $\sigma(\mathbf{z}_{t})$ is a subset of $\mathcal{F}_{t}$ at each $t$; in particular, if $\{\mathcal{F}_{t}\}$ is the natural filtration induced by the observations and Laplacians, then a process $\{\mathbf{z}_{t}\}$ is $\{\mathcal{F}_{t}\}$-adapted if for each $t$ there exists a measurable function $\mathcal{Z}_{t}(\cdot)$ such that $\mathbf{z}_{t}=\mathcal{Z}_{t}\left(\{L_{s},\{\mathbf{y}_{n}(s)\}_{n=1}^{N}\}_{s=0}^{t-1}\right)$.

\subsection{Sensing Model}
\label{subsec:sensmod} Each network agent $n$ sequentially observes an independent and identically distributed (i.i.d.) time-series $\{\mathbf{y}_{n}(t)\}$ of noisy measurements of $\abtheta$, where the distribution $\bmu_{n}^{\abtheta}$ of $\mathbf{y}_{n}(t)$ belongs to a $\btheta$-parameterized exponential family, formalized as follows:

\noindent\begin{assumption}
\label{ass:sensmod} For each $n$, let $\bnu_{n}$ be a $\sigma$-finite measure on $\mathbb{R}^{M_{n}}$. Let $g_{n}:\mathbb{R}^{M_{n}}\mapsto\mathbb{R}^{M}$ be a Borel function such that for all $\btheta\in\mathbb{R}^{M}$ the following expectation exists:
\begin{equation}
\label{ass:sensmod1}\lambda_{n}(\btheta)=\int_{\mathbb{R}^{M_{n}}}e^{\btheta^{\top}g_{n}(\mathbf{y}_{n})}d\bnu_{n}(\mathbf{y}_{n})<\infty.
\end{equation}
Finally, let $\left\{\bmu_{n}^{\btheta}\right\}$, for $\btheta\in\mathbb{R}^{M}$, be the corresponding $\btheta$-parameterized exponential family of distributions on $\mathbb{R}^{M_{n}}$, i.e., for each $\btheta\in\mathbb{R}^{M}$ the probability measure $\bmu_{n}^{\btheta}$ on $\mathbb{R}^{M_{n}}$ is given by the Radon-Nikodym derivative
\begin{equation}
\label{ass:sensmod2}\frac{d\bmu_{n}^{\btheta}}{d\bnu_{n}}(\mathbf{y}_{n})=e^{\left(\btheta^{\top}g_{n}(\mathbf{y}_{n})-\psi_{n}(\btheta)\right)}
\end{equation}
for all $\mathbf{y}_{n}\in\mathbb{R}^{M_{n}}$, where $\psi_{n}(\cdot)$ denotes the function $\psi_{n}(\btheta)=\log\lambda_{n}(\btheta)$.

We assume that each network agent $n$ obtains an $\{\mathcal{F}_{t+1}\}$-adapted independent and identically distributed (i.i.d.) sequence $\{\mathbf{y}_{n}(t)\}$ of observations of the (true) parameter $\btheta^{\ast}$ with distribution $\bmu_{n}(\btheta^{\ast})$, and, for each $t$, $\mathbf{y}_{n}(t)$ is independent of $\mathcal{F}_{t}$. Further, we assume that the observation sequences $\{\mathbf{y}_{n}(t)\}$ and $\{\mathbf{y}_{l}(t)\}$ at any two agents $n$ and $l$ are mutually independent.
\end{assumption}

\begin{remark}
\label{rem:sens_mod} Note that the above formalization enables us to capture very general classes of sensing models. For instance, if the dominating measure $\bnu_{n}$ is the Lebesgue measure, the Radon-Nikodym derivatives in~\eqref{ass:sensmod2} coincide with the standard notion of probability density functions (p.d.f.); hence, standard continuously distributed observation models such as the Gaussian, gamma, beta, etc.~for which p.d.f.'s exist readily fit into our framework by taking $\bnu_{n}$ to be the Lebesgue measure and the Radon-Nikodym derivative by being a p.d.f. Moreover, often in the case of continuous probability distributions, depending on the experimental model and parameterization at hand, the dominating measure need not be exactly the Lebesgue measure but another measure which is absolutely continuous with respect to the Lebesgue measure--a common example is that of \emph{location} families in which the parameter $\btheta$ models translations of a given continuous probability distribution. Differently, by taking $\bnu_{n}$ to be the counting measure, the framework allows us to consider parameterized models with discrete-valued observations in which case the Radon-Nikodym derivatives in~\eqref{ass:sensmod2} correspond to probability mass functions (p.m.f.). This may arise, for instance, in target tracking or source localization applications in which $\btheta$ corresponds to the location of a target in 3-D space and the intensity measuring device (sensor) records only discrete intensity levels rather than registering continuous intensities. More generally, scenarios in which some components of the observation vector $\mathbf{y}_{n}(t)$ are continuous and some discrete, or the distribution of $\mathbf{y}_{n}(t)$ is a mixture of continuous and discrete distributions, may also be modeled by appropriately selecting the dominating measure $\bnu_{n}$. Further, by allowing different observation statistics at different agents~$n$, the framework captures applications with heterogeneous agents and diverse sensing modalities.
\end{remark}

We will also denote by $\mathbf{y}_{t}$ the totality of agent observations at a given time $t$, i.e., $\mathbf{y}_{t}=\vecc(\mathbf{y}_{n}(t))=\left[\mathbf{y}_{1}^{\top}(t),\cdots,\mathbf{y}_{N}^{\top}(t)\right]^{\top}$. For $\btheta\in\mathbb{R}^{M}$ let $\bmu^{\btheta}$ denote the product measure $\bmu_{1}^{\btheta}\otimes\cdots\otimes\bmu_{N}^{\btheta}$ on the product space $\mathbb{R}^{M_{1}}\otimes\cdots\otimes\mathbb{R}^{M_{N}}$, which means the measures $\bmu_{n}$, $n=1\cdots N$, are independent; it is readily seen that $\{\bmu^{\btheta}\}$ is a $\btheta$-parameterized exponential family with respect to (w.r.t.) the product measure $\bnu = \bnu_{1}\otimes\cdots\otimes\bnu_{N}$ and given by the Radon-Nikodym derivatives
\begin{equation}
\label{sensmod3}
\frac{d\bmu^{\btheta}}{d\bnu}(\mathbf{y})=e^{\left(\btheta^{\top}g(\mathbf{y})-\psi(\btheta)\right)},
\end{equation}
where $\mathbf{y}=\vecc(\mathbf{y}_{n})$ denotes a generic element of the product space and the functions $g(\cdot)$ and $\psi(\cdot)$ are given by
\begin{equation}
\label{sensmod5}g(\mathbf{y})=\sum_{n=1}^{N}g_{n}(\mathbf{y}_{n})~~~\mbox{and}~~~\psi(\btheta)=\sum_{n=1}^{N}\psi_{n}(\btheta)
\end{equation}
respectively.

It is readily seen that under Assumption~\ref{ass:sensmod} the global observation sequence $\{\mathbf{y}_{t}\}$ is $\{\mathcal{F}_{t+1}\}$-adapted, with $\mathbf{y}_{t}$ being independent of $\mathcal{F}_{t}$ and distributed as $\bmu^{\btheta^{\ast}}$ (due to mutual independence of the local agent observations) for all $t$.

For most practical agent network applications, each agent observes only a subset of~$M_n$ of the components of the parameter vector, with $M_{n}\ll M$. It is then necessary for the agents to collaborate by means of occasional local inter-agent message exchanges to achieve a reasonable estimate of the parameter $\btheta^{\ast}$. To formalize, while we do not require \emph{local observability} for $\btheta^{\ast}$, we assume that the network sensing model is \emph{globally observable} as follows:
\begin{assumption}
\label{ass:globobs} The network sensing model is globally observable, i.e., we assume $D(\btheta,\btheta^{\prime})>0$ and $D(\btheta^{\prime},\btheta)>0$ for each pair $(\btheta,\btheta^{\prime})$ of parameter values, where $D(\btheta,\btheta^{\prime})$ denotes the Kullback-Leibler divergence between the distributions $\bmu^{\btheta}$ and $\bmu^{\btheta^{\prime}}$, i.e.,
\begin{equation}
\label{ass:globobs1} D(\btheta,\btheta^{\prime})=\int_{\mathbf{y}}\log\left(\frac{d\bmu^{\btheta}}{d\bmu^{\btheta^{\prime}}}(\mathbf{y})\right)d\bmu^{\btheta}(\mathbf{y}).
\end{equation}
\end{assumption}

\subsection{Some preliminaries}
\label{subsec:prel-exp} We state some useful analytical properties associated with the multi-agent sensing model, in particular, the implications of the global observability condition (see Assumption~\ref{ass:globobs}). Most of the listed properties are direct consequences of standard analytical arguments involving statistical exponential families, see, for example,~\cite{Brown:expfam}.
\begin{proposition}
\label{prop:analytic} Let Assumption~\ref{ass:sensmod} hold. Then,
\begin{itemize}
\item[\textup{(1)}] For each $n$, the function $\psi_{n}(\cdot)$ is infinitely differentiable on $\mathbb{R}^{M}$.
\item[\textup{(2)}] For each $n$, let $h_{n}:\mathbb{R}^{M}\mapsto\mathbb{R}^{M}$ be the gradient of $\psi_{n}(\cdot)$, i.e., $h_{n}(\btheta)=\nabla_{\btheta}\psi_{n}(\btheta)$ for all $\btheta\in\mathbb{R}^{M}$. Then\footnote{For a function $f(\btheta)$, $\nabla_{\btheta}f(\btheta)\in\mathbb{R}^{M}$ denotes the vector of partial derivatives, i.e., the $i$-th component of $\nabla_{\btheta}f(\btheta)$ is given by $\frac{\partial f(\btheta)}{\partial\btheta_{i}}$. The Hessian $\nabla^{2}_{\btheta}f(\btheta)\in\mathbb{R}^{M\times M}$ denotes the matrix of second order partial derivatives, whose $i,j$-th entry corresponds to $\frac{\partial^{2} f(\btheta)}{\partial\btheta_{i}\partial\btheta_{j}}$.}
    \begin{equation}
    \label{prop:analytic1}h_{n}(\btheta)=\int_{\mathbf{y}_{n}\in\mathbb{R}^{M_{n}}}g_{n}(\mathbf{y}_{n})d\bmu_{n}^{\btheta}(\mathbf{y}_{n})~~~\forall\btheta\in\mathbb{R}^{M}
    \end{equation}
    and the following inequality (monotonicity) holds for each pair $\left(\btheta,\btheta^{\prime}\right)$ in $\mathbb{R}^{M}$:
    \begin{equation}
    \label{prop:analytic2}\left(\btheta-\btheta^{\prime}\right)^{\top}\left(h_{n}(\btheta)-h_{n}(\btheta^{\prime})\right)\geq 0.
    \end{equation}
\item[\textup{(3)}] If, in addition, Assumption~\ref{ass:globobs} holds, denoting by $h(\cdot)$ the gradient of $\psi(\cdot)$, see~\eqref{sensmod5}, we have the following strict monotonicity
    \begin{equation}
    \label{prop:analytic3}
    \left(\btheta-\btheta^{\prime}\right)^{\top}\left(h(\btheta)-h(\btheta^{\prime})\right)=\sum_{n=1}^{N}\left(\btheta-\btheta^{\prime}\right)^{\top}\left(h_{n}(\btheta)-h_{n}(\btheta^{\prime})\right)>0
    \end{equation}
    for each pair $\left(\btheta,\btheta^{\prime}\right)$ in $\mathbb{R}^{M}$ such that $\btheta\neq\btheta^{\prime}$.
\end{itemize}
\end{proposition}
\begin{proof} The first assertion is an immediate consequence of the fact that the function $\psi_{n}(\btheta)$ associated with the exponential family $\{\bmu_{n}^{\btheta}\}$ is infinitely differentiable on the interior of the natural parameter space (the set on which the expectation in~\eqref{ass:sensmod1} exists), see Theorem 2.2 in~\cite{Brown:expfam}. The second assertion constitutes a well-known property of statistical exponential families (see Corollary 2.5 in~\cite{Brown:expfam}). The same corollary in~\cite{Brown:expfam} asserts that the inequality in~\eqref{prop:analytic2} is strict whenever the measures $\bmu^{\btheta}$ and $\bmu^{\btheta^{\prime}}$ are different for $\btheta\neq\btheta^{\prime}$, the latter being ensured by the positivity of the Kullback-Leibler divergences as in Assumption~\ref{ass:globobs}.
\end{proof}

The next proposition characterizes the information matrices (or Fisher matrices) associated with the sensing model and may be stated as follows (see~\cite{Brown:expfam} for a proof):
\begin{proposition}
\label{prop:inf} Let Assumption~\ref{ass:sensmod} hold. Then,
\begin{itemize}
\item[\textup{(1)}] For each $n$ and $\btheta\in\mathbb{R}^{M}$, let $I_{n}(\btheta)$ denote the Fisher information matrix associated with the exponential family $\{\bmu_{n}^{\btheta}\}$, i.e.,
\begin{equation}
\label{prop:inf1} I_{n}(\btheta)=-\int_{\mathbf{y}_{n}}\left(\nabla^{2}_{\btheta}\frac{d\bmu_{n}^{\btheta}}{d\bnu_{n}}(\mathbf{y}_{n})\right)d\bmu_{n}^{\btheta}(\mathbf{y}_{n}),
\end{equation}
where the expectation integral is to be interpreted entry-wise. Then, $I_{n}(\btheta)$ is positive semidefinite and satisfies $I_{n}(\btheta)=\nabla_{\btheta}\left(h_{n}(\btheta)\right)$ for all $\btheta$, with $h_{n}(\cdot)$ denoting the function in~\eqref{prop:analytic1}.
\item[\textup{(2)}] If, in addition, Assumption~\ref{ass:globobs} holds, the global Fisher information matrix $I(\btheta)$, given by,
\begin{equation}
\label{prop:inf2}I(\btheta)=-\int_{\mathbf{y}}\left(\nabla^{2}_{\btheta}\frac{d\bmu^{\btheta}}{d\bnu}(\mathbf{y})\right)d\bmu^{\btheta}(\mathbf{y}),
\end{equation}
is positive definite and satisfies
\begin{equation}
\label{prop:inf3}
I(\btheta)=\nabla^{2}_{\btheta}h(\btheta)=\sum_{n=1}^{N}\nabla^{2}_{\btheta}h_{n}(\btheta)=\sum_{n=1}^{N}I_{n}(\btheta)
\end{equation}
for all $\btheta\in\mathbb{R}^{M}$.
\end{itemize}
\end{proposition}
For the multi-agent statistical exponential families under consideration, the well-known Cram\'er-Rao characterization holds, and it may be shown that the mean-squared estimation error of any (centralized) estimator based on $t$ sets of observation samples from all the agents is lower bounded by the quantity $t^{-1}I^{-1}(\abtheta)$, where $\abtheta$ denotes the true value of the parameter. Making $t$ tend to $\infty$, the class of asymptotically efficient (optimal) estimators is defined as follows:
\begin{definition}
\label{def:asyeff} An asymptotically efficient estimator of $\abtheta$ is an $\{\mathcal{F}_{t}\}$-adapted sequence $\{\wtheta_{t}\}$, such that $\{\wtheta_{t}\}$ is asymptotically normal with asymptotic covariance $I^{-1}(\abtheta)$, i.e.,
\begin{equation}
\label{def:asyeff1}\sqrt{t+1}\left(\wtheta_{t}-\abtheta\right)\Longrightarrow\mathcal{N}\left(\mathbf{0},I^{-1}(\abtheta)\right),
\end{equation}
where $\Longrightarrow$ and $\mathcal{N}(\cdot,\cdot)$ denote convergence in distribution and the normal distribution respectively.
\end{definition}
\begin{remark}
\label{rem:mle}
\textup{Centralized estimators that are asymptotically efficient for the proposed multi-agent setting may be obtained using now-standard results in point estimation theory. For instance, the (centralized) maximum likelihood estimator is known to achieve asymptotic efficiency; specifically, there exists an $\{\mathcal{F}_{t}\}$-adapted sequence $\{\wtheta_{t}\}$, such that
\begin{equation}
\label{rem:mle1}\wtheta_{t}\in\argmax_{\btheta\in\mathbb{R}^{M}}\left(\sum_{s=0}^{t-1}\log\frac{d\bmu^{\btheta}}{d\bnu}(\mathbf{y}_{s})\right)~~\mbox{a.s. for all $t$},
\end{equation}
and $\{\wtheta_{t}\}$ is asymptotically normal with asymptotic covariance $I^{-1}(\abtheta)$. Note that, apart from being centralized, the maximum likelihood estimator, as implemented above, consists of a batch-form realization. To cope with this, extensive research has focused on the development of time-sequential (but centralized) estimators based on recursively processing the agents' observation data $\mathbf{y}_{t}$; asymptotically efficient recursive centralized estimators of the stochastic approximation type have been developed by several authors, see, for example,~\cite{Sakrison1965efficient,Hasminskii1974estimation,Pfanzagl1974efficient,Stone1975sequentialestimation,Fabian1978efficient}, that are asymptotically efficient.}
\end{remark}

We emphasize that the centralized (recursive or batch) estimators, as discussed above, are based on the availability of the entire set of agent observations at a centralized resource at all times, which further require the global model information (the statistics of the agent exponential families $\{\bmu^{\btheta}_{n}\}$) for all $n$ such that the nonlinear innovation gains driving the recursive estimators may be designed appropriately to achieve asymptotic efficiency. In contrast, the goal of this paper is to develop collaborative distributed asymptotically efficient estimators of $\abtheta$ at each agent $n$ of the network, in which, \begin{inparaenum}\item[(i)] the information is distributed, i.e., at a given instant of time $t$ each agent $n$ has access to its local sensed data $\mathbf{y}_{n}(t)$ only; \item[(ii)] to start with, each agent $n$ is only aware of its local sensing model $\{\bmu^{\btheta}_{n}\}$ only; and, \item[(iii)] the agents may only collaborate by exchanging information over a (sparse) pre-defined communication network, where inter-agent communication and observation sampling occurs at the same rate, i.e., each agent $n$ may only exchange one round of messages with its designated communication neighbors per sampling epoch.\end{inparaenum} To this end, the proposed estimators consist of simultaneous distributed local estimate update and distributed local gain refinement (learning) at each network agent $n$, with closed-loop interaction between the estimation and learning processes. From a technical point of view, in contrast to centralized stochastic approximation based estimators, the estimators developed in the paper are of the distributed nonlinear stochastic approximation type with necessarily mixed time-scale dynamics; the mixed time-scale dynamics arise as a result of suitably crafting the relative intensities of the potentials for local collaboration and local innovation, necessary for achieving asymptotic efficiency. Distributed estimators of mixed time-scale dynamics have been introduced and studied in~\cite{KarMouraRamanan-Est-2008,KarMoura-LinEst-JSTSP-2011}; we refer to them as \textit{consensus + innovations} estimators.

\section{Asymptotically Efficient Distributed Estimator}
\label{sec:asefdistest}
In this section, we provide distributed sequential estimators for $\abtheta$ that are not only consistent but asymptotically optimal, in that, the local asymptotic covariances at each agent coincide with the inverse of the centralized Fisher information rate $I^{-1}(\abtheta)$ associated with the exponential observation statistics in consideration. Other than challenges encountered in the distributed implementation, a major difficulty in obtaining such asymptotically efficient distributed estimators concerns the design of the local estimator or innovation gains (to be made precise later); in particular, in contrast to optimal estimation in linear statistical models~\cite{KarMoura-LinEst-JSTSP-2011,Kar-AdaptiveDistEst-SICON-2012}, in the nonlinear non-Gaussian setting, the innovation gains that achieve asymptotic efficiency are necessarily dependent on the true value $\abtheta$ of the parameter to be estimated. Since the value of $\abtheta$ (and hence the optimal estimator gains) are not available in advance. We propose a distributed estimation approach that involves a distributed online gain learning procedure that proceeds in conjunction with the sequential estimation task. As a result, a somewhat closed-loop interaction occurs between the gain learning and parameter estimation that is reminiscent of the certainty equivalence approach to adaptive estimation and control--although the analysis methodology is significantly different from classical techniques used in adaptive processing, primarily due to the distributed nature of our problem and its mixed time-scale dynamics.

Specifically, the main idea in the proposed distributed estimation methodology is to generate simultaneously two distributed estimators $\{\bx_{n}(t)\}$ and $\{\mathbf{x}_{n}(t)\}$ at each agent $n$; the former, the auxiliary estimate sequences $\{\bx_{n}(t)\}$, are driven by constant (non-adaptive) innovation gains, and, while supposed to be consistent for $\abtheta$, are suboptimal in the sense of asymptotic covariance. The consistent auxiliary estimates are used to generate the sequence of optimal adaptive innovation gains through another online distributed learning procedure; the resulting adaptive gain process is in turn used to drive the evolution of the desired estimate sequences $\{\mathbf{x}_{n}(t)\}$ at each agent $n$, which will be shown to be asymptotically efficient from the asymptotic covariance viewpoint. As will be seen below, we emphasize here that the construction of the auxiliary estimate sequences, the adaptive gain refining, and the generation of the optimal estimators are all executed simultaneously.

\subsection{Algorithms and Assumptions}
\label{subsec:alg} The proposed optimal distributed estimation methodology consists of the following three simultaneous update processes at each agent $n$: \begin{inparaenum}\item[(i)] auxiliary estimate sequence $\{\bx_{n}(t)\}$ generation; \item[(ii)] adaptive gain refinement; and \item[(iii)] optimal estimate sequence $\{\mathbf{x}_{n}(t)\}$ generation.\end{inparaenum} Formally:

\noindent\textbf{Auxiliary Estimate Generation}: Each agent $n$ maintains an $\{\mathcal{F}_{t}\}$-adapted $\mathbb{R}^{M}$-valued estimate sequence $\{\bx_{n}(t)\}$ for $\abtheta$, recursively updated in a distributed fashion as follows:
\begin{equation}
\label{aux:1}
\bx_{n}(t+1)=\bx_{n}(t)-\beta_{t}\sum_{l\in\Omega_{n}(t)}\left(\bx_{n}(t)-\bx_{l}(t)\right)+\alpha_{t}\left(g_{n}(\mathbf{y}_{n}(t))-h_{n}(\bx_{n}(t))\right),
\end{equation}
where $\{\beta_{t}\}$ and $\{\alpha_{t}\}$ correspond to appropriate time-varying weighting factors for the agreement (consensus) and innovation (new observation) potentials, respectively, whereas, $\Omega_{n}(t)$ denotes the $\{\mathcal{F}_{t+1}\}$-adapted time-varying random neighborhood of agent $n$ at time $t$.

\noindent\textbf{Optimal Estimate Generation}: In addition, each agent $n$ generates an optimal (or refined) estimate sequence $\{\mathbf{x}_{n}(t)\}$, which is also $\{\mathcal{F}_{t}\}$-adapted and evolves as
\begin{equation}
\label{opt:1} \mathbf{x}_{n}(t+1)=\mathbf{x}_{n}(t)-\beta_{t}\sum_{l\in\Omega_{n}(t)}\left(\mathbf{x}_{n}(t)-\mathbf{x}_{l}(t)\right)+\alpha_{t}K_{n}(t)\left(g_{n}(\mathbf{y}_{n}(t))-h_{n}(\mathbf{x}_{n}(t))\right).
\end{equation}
Note that the key difference between the estimate updates in~\eqref{aux:1} and~\eqref{opt:1} is in the use of adaptive (time-varying) gains $K_{n}(t)$ in the innovation part in the latter, as opposed to static gains in the former. Specifically, the adaptive gain sequence $\{K_{n}(t)\}$ at an agent $n$ is an $\{\mathcal{F}_{t}\}$-adapted $\mathbb{R}^{M\times M}$-valued process which is generated according to a distributed learning process as follows.

\noindent\textbf{Adaptive Gain Refinement}: The $\{\mathcal{F}_{t}\}$-adapted gain sequence $\{K_{n}(t)\}$ at an agent $n$ is generated according to a distributed learning process, driven by the auxiliary estimates $\{\bx_{n}(t)\}$ obtained in~\eqref{aux:1}, as follows:
\begin{equation}
\label{gain1}K_{n}(t)=\left(G_{n}(t)+\varphi_{t}I_{M}\right)^{-1}~~~\forall n,
\end{equation}
where, $\{\varphi_{t}\}$ is a deterministic sequence of positive numbers such that $\varphi_{t}\rightarrow 0$ as $t\rightarrow\infty$, and, each agent $n$ maintains another $\{\mathcal{F}_{t}\}$-adapted $\mathbb{S}_{+}^{M}$-valued process $\{G_{n}(t)\}$ evolving in a distributed fashion as
\begin{equation}
\label{gain2}G_{n}(t+1)=G_{n}(t)-\beta_{t}\left(G_{n}(t)-G_{l}(t)\right)+\alpha_{t}\left(I_{n}(\bx_{n}(t))-G_{n}(t)\right)
\end{equation}
for all $t$, with some positive semidefinite initial condition $G_{n}(0)$ and $I_{n}(\cdot)$ denoting the local Fisher information matrix, see~\eqref{prop:inf1}.

\begin{assumption}
\label{ass:conn} The $\{\mathcal{F}_{t+1}\}$-adapted sequence $\{L_{t}\}$ of communication network Laplacians (modeling the agent communication neighborhoods $\Omega_{n}(t)$-s at each time $t$) is temporally i.i.d. with $L_{t}$ being independent of $\mathcal{F}_{t}$ for each $t$. Further, the sequence $\{L_{t}\}$ is connected on the average, i.e., $\lambda_{2}(\overline{L})>0$, where $\overline{L}=\East[L_{t}]$ denotes the mean Laplacian.
\end{assumption}
\begin{assumption}
\label{ass:weight} The weight sequences $\{\beta_{t}\}$ and $\{\alpha_{t}\}$ satisfy
\begin{equation}
\label{weight}
\alpha_{t}=\frac{1}{(t+1)}~~~\mbox{and}~~~\beta_{t}=\frac{b}{(t+1)^{\tau_{2}}},
\end{equation}
where $b>0$ and $0<\tau_{2}< 1/2$.

Further, the sequence $\{\varphi_{t}\}$ in~\eqref{gain1} satisfies
\begin{equation}
\label{ass:weight3}
\lim_{\tri}(t+1)^{\mu_{2}}\varphi_{t}=0
\end{equation}
for some positive constant $\mu_{2}$.
\end{assumption}

The following weak linear growth condition on the functions $h_{n}(\cdot)$ driving the (nonlinear) innovations in~\eqref{aux:1}-\eqref{opt:1} will be assumed:

\begin{assumption}
\label{ass:lingrowth} For each $\btheta\in\mathbb{R}^{M}$, there exist positive constants $c_{1}^{\btheta}$ and $c_{2}^{\btheta}$, such that, for each $n$, function $h_{n}(\cdot)$ in~\eqref{prop:analytic1} satisfies the local linear growth condition,
\begin{equation}
\label{ass:lingrowth1} \left\|h_{n}(\btheta^{\prime})-h_{n}(\btheta)\right\|\leq c_{1}^{\btheta}\left\|\btheta^{\prime}-\btheta\right\|+c_{2}^{\btheta},
\end{equation}
for all $\btheta^{\prime}\in\mathbb{R}^{M}$.
\end{assumption}


\subsection{Main Results}
\label{subsec:mainres} We formally state the main results of the paper, the proofs appearing in Section~\ref{sec:proof_main_res}.

\begin{theorem}
\label{th:estcons} Let Assumptions~\ref{ass:globobs},\ref{ass:conn},\ref{ass:lingrowth} and~\ref{ass:weight} hold. Then, for each $n$ the estimate sequence $\{\mathbf{x}_{n}(t)\}$ is strongly consistent. In particular, we have
\begin{equation}
\label{th:estcons1}
\Past\left(\lim_{t\rightarrow\infty}(t+1)^{\tau}\left\|\mathbf{x}_{n}(t)-\mathbf{\theta}^{\ast}\right\|=0\right)=1
\end{equation}
for each $n$ and $\tau\in [0,1/2)$.
\end{theorem}

The consistency in Theorem~\ref{th:estcons} is order optimal in that~\eqref{th:estcons1} fails to hold with an exponent $\tau\geq 1/2$ for any (including centralized) estimation procedure.


The next result concerns the asymptotic efficiency of the estimates generated by the proposed distributed scheme.
\begin{theorem}
\label{th:estn} Let Assumptions~\ref{ass:globobs},\ref{ass:conn},\ref{ass:lingrowth} and~\ref{ass:weight} hold. Then, for each $n$ we have
\begin{equation}
\label{th:estn200}
\sqrt(t+1)\left(\mathbf{x}_{n}(t)-\mathbf{\theta}^{\ast}\right)\Longrightarrow\mathcal{N}\left(\mathbf{0},I^{-1}(\abtheta)\right),
\end{equation}
where $\mathcal{N}(\cdot,\cdot)$ and $\Longrightarrow$ denote the Gaussian distribution and weak convergence, respectively.
\end{theorem}

\noindent{\bf Discussion}: We discuss some key aspects of the distributed recursive scheme~\eqref{aux:1}-\eqref{gain2}. First, note that the distributed

\section{A Generic Consistent Distributed Estimator}
\label{sec:genconsest} With a view to understanding the asymptotic behavior of the auxiliary estimate processes $\{\bx_{n}(t)\}$, $n=1,\cdots,N$, introduced in Section~\ref{subsec:alg}, see~\eqref{aux:1}, we study a somewhat more general class of distributed estimate processes with time-varying local innovation gains. Other than establishing consistency of these estimates (see Theorem~\ref{th:genest}), we obtain pathwise convergence rate asymptotics of the estimate processes to $\abtheta$ (see Theorem~\ref{th:genrate}). These latter convergence rate results will be used to analyze the impact of the auxiliary estimates in the adaptive gain computation~\eqref{gain1}-\eqref{gain2}.
\begin{theorem}
\label{th:genest} For each $n$, let $\{\mathbf{z}_{n}(t)\}$ be an $\mathbb{R}^{M}$-valued $\{\mathcal{F}_{t}\}$-adapted process (estimator) evolving as follows:
\begin{equation}
\label{th:genest1}\mathbf{z}_{n}(t+1)=\mathbf{z}_{n}(t)-\beta_{t}\sum_{n\in\Omega_{n}(t)}\left(\mathbf{z}_{n}(t)-\mathbf{z}_{l}(t)\right)+\alpha_{t}K_{n}(t)\left(g_{n}(\mathbf{y}_{n}(t)-h_{n}(\mathbf{z}_{n}(t))\right).
\end{equation}
Suppose Assumptions~\ref{ass:sensmod},\ref{ass:lingrowth} and \ref{ass:conn} on the network system model hold, and the weight sequences $\{\beta_{t}\}$ and $\{\alpha_{t}\}$ satisfy Assumption~\ref{ass:weight}. Additionally, let the matrix gain processes $\{K_{n}(t)\}$ be $\mathbb{S}_{+}^{M}$-valued $\{\mathcal{F}_{t}\}$-adapted, and there exist a positive definite matrix $\mathcal{K}$ and a constant $\tau_{3}>0$, such that the gain processes $\{K_{n}(t)\}$ converge uniformly to $\mathcal{K}$ at rate $\tau_{3}$, i.e., for each $\delta>0$, there exists a deterministic time $t_{\delta}$, such that for all $n$
\begin{equation}
\label{lm:bg1}\Past\left(\sup_{t\geq t_{\delta}}(t+1)^{\tau_{3}}\left\|K_{n}(t)-\mathcal{K}\right\|\leq\delta\right)=1.
\end{equation}
Then, for each $n$, $\{\mathbf{z}_{n}(t)\}$ is a consistent estimator of $\theta^{\ast}$, i.e., $\mathbf{z}_{n}(t)\rightarrow\btheta^{\ast}$ as $t\rightarrow\infty$ a.s.
\end{theorem}

The proof of Theorem~\ref{th:genest} is accomplished in steps, the key intermediate ingredients being Lemma~\ref{lm:bg} and Proposition~\ref{prop:Lg} concerning the boundedness of the processes $\{\mathbf{z}_{n}(t)\}$, $n=1,\cdots,N$, and a Lyapunov type-construction, respectively.

\begin{lemma}
\label{lm:bg} Let the hypotheses of Theorem~\ref{th:genest} hold.
Then, for each $n$, the process $\{\mathbf{z}_{n}(t)\}$ is bounded a.s., i.e.,
\begin{equation}
\label{lm:bg2}\Past\left(\sup_{t\geq 0}\left\|\mathbf{z}_{n}(t)\right\|<\infty\right)=1.
\end{equation}
\end{lemma}
\begin{proof} Let $\wz_{n}(t)=\mathbf{z}_{n}(t)-\abtheta$ and denote by $\mathbf{z}_{t}$, $\wz_{t}$ and $\abbtheta$ the $\mathbb{R}^{NM}$-valued $\vecc\left(\mathbf{z}_{n}(t)\right)$, $\vecc\left(\wz_{n}(t)\right)$, and $\mathbf{1}_{N}\otimes\abtheta$, respectively. Noting that $\left(L_{t}\otimes I_{M}\right)\left(\mathbf{1}_{N}\otimes\abtheta\right)=\mathbf{0}$, the process $\{\wz_{t}\}$ is seen to satisfy
\begin{equation}
\label{lm:bg3}
\wz_{t+1}=\wz_{t}-\beta_{t}\left(L_{t}\otimes I_{M}\right)\wz_{t}-\alpha_{t}\bK_{t}\left(\bh(\mathbf{z}_{t})-\bh(\abbtheta)\right)+\alpha_{t}\bK_{t}\left(\bg(\mathbf{y}_{t})-\bh(\abbtheta)\right),
\end{equation}
where
\begin{equation}
\label{lm:bg4}
\bh(\mathbf{z}_{t})=\vecc\left(h_{n}(\mathbf{z}_{n}(t))\right),~~\bh(\abbtheta)=\vecc\left(h_{n}(\abtheta)\right),~~\bg(\mathbf{y}_{t})=\vecc\left(g_{n}(\mathbf{y}_{n}(t))\right),
\end{equation}
and $\bK_{t}=\ndiag(K_{n}(t))$. Note that, by hypothesis, $\bK_{t}\in\mathbb{S}_{++}^{NM}$ and define the $\mathbb{R}_{+}$-valued $\{\mathcal{F}_{t}\}$-adapted process $\{V_{t}\}$ by
\begin{equation}
\label{lm:bg5}V_{t}=\wz_{t}^{\top}\left(I_{N}\otimes\mathcal{K}^{-1}\right)\wz_{t}
\end{equation}
for all $t$. Note that by~\eqref{lm:bg3} we obtain
\begin{align}
\label{lm:bg84}
\left(I_{N}\otimes\mathcal{K}\right)^{-1}\wz_{t+1}=\left(I_{N}\otimes\mathcal{K}^{-1}\right)\wz_{t}-\beta_{t}\left(L_{t}\otimes\mathcal{K}^{-1}\right)\wz_{t}\\-\alpha_{t}\left(I_{N}\otimes\mathcal{K}\right)^{-1}\bK_{t}\left(\bh(\mathbf{z}_{t})-\bh(\abbtheta)\right)
+ \alpha_{t}\left(I_{N}\otimes\mathcal{K}\right)^{-1}\bK_{t}\left(\bg(\mathbf{y}_{t})-\bh(\abbtheta)\right).
\end{align}
By~\eqref{prop:analytic1} we have for all $t\geq 0$
\begin{equation}
\label{lm:bg6}\East\left[\bg(\mathbf{y}_{t})-\bh(\abbtheta)\right]=\mathbf{0},
\end{equation}
and using the temporal independence of the Laplacian sequence we obtain
\begin{align}
\nonumber
\East\left[V_{t+1}~|~\mathcal{F}_{t}\right] &=V_{t}-2\beta_{t}\wz_{t}^{\top}\left(\OL\otimes\mathcal{K}^{-1}\right)\wz_{t}-2\alpha_{t}\wz_{t}^{\top}\left(I_{N}\otimes\mathcal{K}^{-1}\right) K_{t}\left(\bh(\mathbf{z}_{t})-\bh(\abbtheta)\right)
\\
\nonumber
 & +\beta_{t}^{2}\wz_{t}^{\top}\East\left[\left(\OL\otimes I_{M}\right)\left(I_{N}\otimes\mathcal{K}^{-1}\right)\left(\OL\otimes I_{M}\right)\right]\wz_{t}
 \\
 \nonumber
 & +2\alpha_{t}\beta_{t}\wz_{t}^{\top}\left(\OL\otimes I_{M}\right)\left(I_{N}\otimes\mathcal{K}^{-1}\right)K_{t}\left(\bh(\mathbf{z}_{t}-\bh(\abbtheta)\right)\\
& +\alpha_{t}^{2}\left(\bh(\mathbf{z}_{t}-\bh(\abtheta)\right)^{\top}K_{t}\left(I_{N}\otimes\mathcal{K}^{-1}\right)K_{t} \left(\bh(\mathbf{z}_{t}-\bh(\abbtheta)\right)
\\
\label{lm:bg7}
& +\alpha_{t}^{2}\East\left[\left(\bg(\mathbf{y}_{t})-\bh(\abbtheta)\right)^{\top}K_{t}\left(I_{N}\otimes\mathcal{K}^{-1}\right) K_{t}\left(\bg(\mathbf{y}_{t})-\bh(\abbtheta)\right)\right]
\end{align}
for all $t\geq 0$.

Recall the definition of consensus subspace in Definition~\ref{def:consspace} and note that by using the properties of the Laplacian $\OL$ and matrix Kronecker products we have
\begin{equation}
\label{lm:bg201}
\wz_{t}^{\top}\left(\OL\otimes\mathcal{K}^{-1}\right)\wz_{t}=\left(\wz_{t}\right)_{\PC}^{\top}\left(\OL\otimes\mathcal{K}^{-1}\right)\left(\wz_{t}\right)_{\PC}\geq\lambda_{2}(\OL)\lambda_{1}\left(\mathcal{K}^{-1}\right)\left\|\left(\wz_{t}\right)_{\PC}\right\|^{2}
\end{equation}
for all $t\geq 0$, where $\lambda_{1}\left(\mathcal{K}^{-1}\right)>0$ denotes the smallest eigenvalue of the positive definite matrix $\mathcal{K}^{-1}$.

Now consider the inequality
\begin{align}
\label{lm:bg8}\wz_{t}^{\top}\left(\bh(\mathbf{z}_{t})-\bh(\abbtheta)\right) = \sum_{n=1}^{N}\left(\mathbf{z}_{n}(t)-\abtheta\right)^{\top}\left(h_{n}(\mathbf{z}_{n}(t))-h_{n}(\abtheta)\right)\geq 0
\end{align}
(where the non-negativity of the terms in the summation follows from Proposition~\ref{prop:analytic}), and note that, by Assumption~\ref{ass:lingrowth} and hypothesis~\eqref{lm:bg1}, there exist positive constants $c_{1}$ and $t_{1}$ large enough such that
\begin{align}
\label{lm:bg9}\wz_{t}^{\top}\left(I_{N}\otimes\mathcal{K}^{-1}\right)K_{t}\left(\bh(\mathbf{z}_{t})-\bh(\abbtheta)\right)\\ \geq \wz_{t}^{\top}\left(\bh(\mathbf{z}_{t})-\bh(\abbtheta)\right)-\left|\wz_{t}^{\top}\left(I_{N}\otimes\mathcal{K}^{-1}\right)\left(K_{t}-\left(I_{N}\otimes\mathcal{K}^{-1}\right)\right)\left(\bh(\mathbf{z}_{t})-\bh(\abbtheta)\right)\right|\\
\geq -\left\|\wz_{t}\right\|\left\|I_{N}\otimes\mathcal{K}^{-1}\right\|\left\|K_{t}-\left(I_{N}\otimes\mathcal{K}^{-1}\right)\right\|\left\|\bh(\mathbf{z}_{t})-\bh(\abbtheta)\right\|\\
\geq -c_{1}\left(1/(t+1)^{\tau_{3}}\right)\left(1+\left\|\wz_{t}\right\|^{2}\right)
\end{align}
for all $t\geq t_{1}$, where we also use the inequality $\|\wz_{t}\|\leq \|\wz_{t}\|^{2}+1$. Similarly, by invoking the boundedness of the matrices involved and the linear growth condition on the $h_{n}(\cdot)$-s and making $c_{1}$ and $t_{1}$ larger if necessary, we obtain the following sequence of inequalities for all $t\geq t_{1}$:
\begin{align}
\label{lm:bg10}
\wz_{t}^{\top}\East\left[\left(\OL\otimes I_{M}\right)\left(I_{N}\otimes\mathcal{K}^{-1}\right)\left(\OL\otimes I_{M}\right)\right]\wz_{t}\\ =\left(\wz_{t}\right)_{\PC}^{\top}\East\left[\left(\OL\otimes I_{M}\right)\left(I_{N}\otimes\mathcal{K}^{-1}\right)\left(\OL\otimes I_{M}\right)\right]\left(\wz_{t}\right)_{\PC}\leq c_{1}\left\|\left(\wz_{t}\right)_{\PC}\right\|^{2},
\end{align}
\begin{equation}
\label{lm:bg11}\wz_{t}^{\top}\left(\OL\otimes I_{M}\right)\left(I_{N}\otimes\mathcal{K}^{-1}\right)K_{t}\left(\bh(\mathbf{z}_{t})-\bh(\abbtheta)\right)\leq c_{1}\left(1+\left\|\wz_{t}\right\|^{2}\right),
\end{equation}
\begin{equation}
\label{lm:bg12}\left(\bh(\mathbf{z}_{t}-\bh(\abtheta)\right)^{\top}K_{t}\left(I_{N}\otimes\mathcal{K}^{-1}\right)K_{t}\left(\bh(\mathbf{z}_{t}-\bh(\abbtheta)\right)\leq c_{1}\left(1+\left\|\wz_{t}\right\|^{2}\right),
\end{equation}
and
\begin{equation}
\label{lm:bg13}\East\left[\left(\bg(\mathbf{y}_{t})-\bh(\abbtheta)\right)^{\top}K_{t}\left(I_{N}\otimes\mathcal{K}^{-1}\right)K_{t}\left(\bg(\mathbf{y}_{t})-\bh(\abbtheta)\right)\right]\leq c_{1},
\end{equation}
where the last inequality uses the fact that $\bg(\mathbf{y}_{t})$ possesses moments of all orders due to the exponential statistics.

Noting that there exist positive constants $c_{2}$ and $c_{3}$ such that
\begin{equation}
\label{lm:bg14}
c_{2}\left\|\wz_{t}\right\|^{2}\leq \wz_{t}^{\top}\left(I_{N}\otimes\mathcal{K}^{-1}\right)\wz_{t}=V_{t}\leq c_{3}\left\|\wz_{t}\right\|^{2}
\end{equation}
for all $t$, by~\eqref{lm:bg7}-\eqref{lm:bg13} we have for all $t\geq t_{1}$
\begin{align}
\label{lm:bg15}
\East\left[V_{t+1}~|~\mathcal{F}_{t}\right]\leq\left(1+c_{4}\alpha_{t}\left(\frac{1}{(t+1)^{\tau_{3}}}+\beta_{t}+ \alpha_{t}\right)\right)V_{t}
\\
-c_{5}\left(\beta_{t}-\beta_{t}^{2}\right)\left\|\left(\wz_{t}\right)_{\PC}\right\|^{2}+c_{6}\left(\frac{\alpha_{t}}{(t+1)^{\tau_{3}}}+\alpha_{t}\beta_{t}+\alpha_{t}^{2}\right)
\end{align}
for some positive constants $c_{4}$, $c_{5}$ and $c_{6}$. Since $\beta_{t}\rightarrow 0$ as $t\rightarrow\infty$ by~\eqref{weight}, we may choose $t_{2}$ large enough (larger than $t_{1}$) such that $\left(\beta_{t}-\beta_{t}^{2}\right)\geq 0$ for all $t\geq t_{2}$. Further, the hypotheses on the weight sequences~\eqref{weight} confirm the existence of constants $\tau_{4}$ and $\tau_{5}$ strictly greater than 1, and positive constants $c_{7}$ and $c_{8}$, such that
\begin{equation}
\label{lm:bg16}
c_{4}\alpha_{t}\left(\frac{1}{(t+1)^{\tau_{3}}}+\beta_{t}+\alpha_{t}\right)\leq \frac{c_{7}}{(t+1)^{\tau_{4}}}=\gamma_{t}
\end{equation}
and
\begin{equation}
\label{lm:bg17}
c_{6}\left(\frac{\alpha_{t}}{(t+1)^{\tau_{3}}}+\alpha_{t}\beta_{t}+\alpha_{t}^{2}\right)\leq \frac{c_{8}}{(t+1)^{\tau_{5}}}=\gamma^{\prime}_{t}
\end{equation}
for all $t\geq t_{2}$ (by making $t_{2}$ larger if necessary). By the above construction we then obtain
\begin{equation}
\label{lm:bg18}\East\left[V_{t+1}~|~\mathcal{F}_{t}\right]\leq\left(1+\gamma_{t}\right)V_{t}+\gamma^{\prime}_{t}
\end{equation}
for all $t\geq t_{2}$ with the positive weight sequences $\{\gamma_{t}\}$ and $\{\gamma^{\prime}_{t}\}$ being summable, i.e.,
\begin{equation}
\label{lm:bg19}\sum_{t\geq 0}\gamma_{t}<\infty~~\mbox{and}~~\sum_{t\geq 0}\gamma^{\prime}_{t}<\infty.
\end{equation}
Note that, by~\eqref{lm:bg19}, the product $\prod_{s=t}^{\infty}(1+\gamma_{s})$ exists for all $t$, and define by $\{W_{t}\}$ the $\mathbb{R}_{+}$-valued $\{\mathcal{F}_{t}\}$-adapted process such that
\begin{equation}
\label{lm:bg20}
W_{t}=\left(\prod_{s=t}^{\infty}(1+\gamma_{s})\right)V_{t}+\sum_{s=t}^{\infty}\gamma^{\prime}_{s},~~~\forall t.
\end{equation}
By~\eqref{lm:bg18}, the process $\{W_{t}\}$ may be shown to satisfy
\begin{equation}
\label{lm:bg21}\East\left[W_{t+1}~|~\mathcal{F}_{t}\right]\leq W_{t}
\end{equation}
for all $t\geq t_{2}$. Being a non-negative supermartingale the process $\{W_{t}\}$ converges a.s. to a bounded random variable $W^{\ast}$ as $t\rightarrow\infty$. It then follows readily by~\eqref{lm:bg20} that $V_{t}\rightarrow W^{\ast}$ a.s. as $t\rightarrow\infty$. In particular, we conclude that the process $\{V_{t}\}$ is bounded a.s., which establishes the desired boundedness of the sequences $\{\mathbf{z}_{n}(t)\}$ for all $n$.
\end{proof}

The following useful convergence may be extracted as a corollary to Lemma~\ref{lm:bg}.
\begin{corollary}
\label{corr:bg} Under the hypotheses of Lemma~\ref{lm:bg}, there exists a finite random variable $V^{\ast}$ such that $V_{t}\rightarrow V^{\ast}$ a.s. as $t\rightarrow\infty$, where $V_{t}=\wz_{t}^{\top}\left(I_{N}\otimes\mathcal{K}^{-1}\right)\wz_{t}$ as in~\eqref{lm:bg5}.
\end{corollary}

The following Lyapunov-type construction, whose proof is relegated to Appendix~\ref{sec:app1}, will be critical to the subsequent development.
\begin{proposition}
\label{prop:Lg} Let $\Vap\in (0,1)$ and $\Gamma_{\Vap}$ denote the set
\begin{equation}
\label{prop:Lg1} \Gamma_{\Vap}=\left\{\mathbf{z}\in\mathbb{R}^{NM}~:~\Vap\leq\left\|\mathbf{z}-\abbtheta\right\|\leq 1/\Vap\right\}.
\end{equation}
For each $t\geq 0$, denote by $\mathcal{H}_{t}:\mathbb{R}^{NM}\mapsto\mathbb{R}$ the function given by
\begin{equation}
\label{prop:Lg3}
\mathcal{H}_{t}(\mathbf{z})=\frac{b_{\beta}\beta_{t}}{\alpha_{t}}\left(\mathbf{z}-\abbtheta\right)^{\top}\left(\OL\otimes\mathcal{K}^{-1}\right)\left(\mathbf{z}-\abbtheta\right)+\left(\mathbf{z}-\abbtheta\right)^{\top}\left(\bh(\mathbf{z})-\bh(\abbtheta)\right)
\end{equation}
for all $\mathbf{z}\in\mathbb{R}^{NM}$, where the matrix $\mathcal{K}^{-1}\in\mathbb{S}_{++}^{M}$ and $b_{\beta}>0$ is a constant. Then, there exist $t_{\Vap}>0$ and a constant $\bc_{\Vap}>0$ such that for all $t\geq t_{\Vap}$
\begin{equation}
\label{prop:Lg4}\mathcal{H}_{t}(\mathbf{z})\geq\bc_{\Vap}\left\|\mathbf{z}-\abbtheta\right\|^{2},~~~\forall \mathbf{z}\in\Gamma_{\Vap}.
\end{equation}
\end{proposition}

We now complete the proof of Theorem~\ref{th:genest}.
\begin{proof}[Proof of Theorem~\ref{th:genest}] In what follows we use the notation and definitions formulated in the proof of Lemma~\ref{lm:bg}. Let us consider $\Vap\in (0,1)$ and let $\rho_{\Vap}$ denote the $\{\mathcal{F}_{t}\}$ stopping time
\begin{equation}
\label{th:genest10}
\rho_{\Vap}=\inf\left\{t\geq 0~:~\mathbf{z}_{t}\notin\Gamma_{\Vap}\right\},
\end{equation}
where $\Gamma_{\Vap}$ is defined in~\eqref{prop:Lg1}. Let $\{V_{t}\}$ be the $\{\mathcal{F}_{t}\}$-adapted process defined in~\eqref{lm:bg5} and denote by $\{V^{\Vap}_{t}\}$ the stopped process
\begin{equation}
\label{th:genest11}V^{\Vap}_{t}=V_{t\wedge\rho_{\Vap}},~~~\forall t,
\end{equation}
which is readily seen to be $\{\mathcal{F}_{t}\}$ adapted. Noting that
\begin{equation}
\label{th:genest12}V^{\Vap}_{t+1}=V_{t+1}\mathbb{I}\left(\rho_{\Vap}>t\right)+V_{\rho_{\Vap}}\mathbb{I}\left(\rho_{\Vap}\leq t\right)
\end{equation}
and the fact that the indicator function $\mathbb{I}\left(\rho_{\Vap}>t\right)$ and the random variable $V_{\rho_{\Vap}}\mathbb{I}\left(\rho_{\Vap}\leq t\right)$ are adapted to $\mathcal{F}_{t}$ for all $t$ ($\rho_{\Vap}$ being an $\{\mathcal{F}_{t}\}$ stopping time), we have
\begin{equation}
\label{th:genest13}\East\left[V^{\Vap}_{t+1}~|~\mathcal{F}_{t}\right]=\East\left[V_{t+1}~|~\mathcal{F}_{t}\right]\mathbb{I}\left(\rho_{\Vap}>t\right)+V_{\rho_{\Vap}}\mathbb{I}\left(\rho_{\Vap}\leq t\right)
\end{equation}
for all $t$.

Recall the function $\mathcal{H}_{t}(\cdot)$ defined in~\eqref{prop:Lg3}; setting $b_{\beta}=1/2$ in the definition of $\mathcal{H}_{t}(\cdot)$ we obtain
\begin{align}
\label{th:genest200}2\beta_{t}\wz_{t}^{\top}\left(\OL\otimes\mathcal{K}^{-1}\right)\wz_{t}+2\alpha_{t}\wz_{t}^{\top}\left(I_{N}\otimes\mathcal{K}^{-1}\right)K_{t}\left(\bh(\mathbf{z}_{t})-\bh(\abbtheta)\right)\\
=2\alpha_{t}\mathcal{H}_{t}(\mathbf{z}_{t})+\beta_{t}\wz_{t}^{\top}\left(\OL\otimes\mathcal{K}^{-1}\right)\wz_{t}\\+2\alpha_{t}\wz_{t}^{\top}\left(I_{N}\otimes\mathcal{K}^{-1}\right)\left(K_{t}-\left(I_{N}\otimes\mathcal{K}^{-1}\right)\right)\left(\bh(\mathbf{z}_{t})-\bh(\abbtheta)\right).
\end{align}
A slight rearrangement of the terms in the expansion~\eqref{lm:bg7} then yields
\begin{align}
\label{th:genest14}
\East\left[V_{t+1}~|~\mathcal{F}_{t}\right]=V_{t}-2\alpha_{t}\mathcal{H}_{t}(\mathbf{z}_{t})-\beta_{t}\wz_{t}^{\top}\left(\OL\otimes\mathcal{K}^{-1}\right)\wz_{t}\\-2\alpha_{t}\wz_{t}^{\top}\left(I_{N}\otimes\mathcal{K}^{-1}\right)\left(K_{t}-\left(I_{N}\otimes\mathcal{K}^{-1}\right)\right)\left(\bh(\mathbf{z}_{t})-\bh(\abbtheta)\right)\\
+\beta_{t}^{2}\wz_{t}^{\top}\East\left[\left(\OL\otimes I_{M}\right)\left(I_{N}\otimes\mathcal{K}^{-1}\right)\left(\OL\otimes I_{M}\right)\right]\wz_{t}\\+2\alpha_{t}\beta_{t}\wz_{t}^{\top}\left(\OL\otimes I_{M}\right)\left(I_{N}\otimes\mathcal{K}^{-1}\right)K_{t}\left(\bh(\mathbf{z}_{t}-\bh(\abbtheta)\right)\\
+\alpha_{t}^{2}\left(\bh(\mathbf{z}_{t}-\bh(\abtheta)\right)^{\top}K_{t}\left(I_{N}\otimes\mathcal{K}^{-1}\right)K_{t}\left(\bh(\mathbf{z}_{t}-\bh(\abbtheta)\right)\\+\alpha_{t}^{2}\East\left[\left(\bg(\mathbf{y}_{t})-\bh(\abbtheta)\right)^{\top}K_{t}\left(I_{N}\otimes\mathcal{K}^{-1}\right)K_{t}\left(\bg(\mathbf{y}_{t})-\bh(\abbtheta)\right)\right]
\end{align}
for all $t\geq 0$, where $\mathcal{H}_{t}(\cdot)$ is defined in~\eqref{prop:Lg3}. The inequalities in~\eqref{lm:bg201}-\eqref{lm:bg14} then show that there exist positive constants $b_{1}$, $b_{2}$ and $b_{3}$, and a deterministic time $t_{1}$ (large enough), such that,
\begin{align}
\label{th:genest15}\East\left[V_{t+1}~|~\mathcal{F}_{t}\right]\leq\left(1+b_{1}\left(\alpha_{t}(t+1)^{-\tau_{3}}+\alpha^{2}_{t}+\alpha_{t}\beta_{t}\right)\right)V_{t}-2\alpha_{t}\mathcal{H}_{t}(\mathbf{z}_{t})\\
-b_{2}\left(\beta_{t}-\beta_{t}^{2}\right)\left\|(\wz_{t})_{\PC}\right\|^{2}+b_{3}\left(\alpha_{t}(t+1)^{-\tau_{3}}+\alpha^{2}_{t}+\alpha_{t}\beta_{t}\right)
\end{align}
for all $t\geq t_{1}$. Note that, by definition, on the event $\{\rho_{\Vap}>t\}$ we have $\mathbf{z}_{t}\in\Gamma_{\Vap}$, and hence, an immediate application of Proposition~\ref{prop:Lg} establishes the existence of a positive constant $\bc_{\Vap}$ and a large enough deterministic time $t_{\Vap}>0$, such that,
\begin{equation}
\label{th:genest16}\mathcal{H}_{t}(\mathbf{z}_{t})\mathbb{I}\left(\rho_{\Vap}>t\right)\geq\bc_{\Vap}\|\wz_{t}\|^{2}\mathbb{I}\left(\rho_{\Vap}>t\right)
\end{equation}
for all $t\geq t_{\Vap}$. By~\eqref{lm:bg14} and~\eqref{th:genest14}-\eqref{th:genest15} and making $t_{\Vap}$ larger if necessary, it then follows that there exist a constant $b_{4}(\Vap)>0$ such that
\begin{align}
\label{th:genest17}\East\left[V_{t+1}~|~\mathcal{F}_{t}\right]\mathbb{I}\left(\rho_{\Vap}>t\right)\leq\left[\left(1-b_{4}(\Vap)\alpha_{t}+b_{1}\left(\alpha_{t}(t+1)^{-\tau_{3}}+\alpha^{2}_{t}+\alpha_{t}\beta_{t}\right)\right)V_{t}\right.\\
\left.-b_{2}\left(\beta_{t}-\beta_{t}^{2}\right)\left\|(\wz_{t})_{\PC}\right\|^{2}+b_{3}\left(\alpha_{t}(t+1)^{-\tau_{3}}+\alpha^{2}_{t}+\alpha_{t}\beta_{t}\right)\right]\mathbb{I}\left(\rho_{\Vap}>t\right)
\end{align}
for all $t\geq t_{\Vap}$. Since $\alpha_{t}\rightarrow 0$ and $\beta_{t}\rightarrow 0$ as $t\rightarrow\infty$, by choosing $t_{\Vap}$ large enough we may assert
\begin{equation}
\label{th:genest20}\beta_{t}-\beta_{t}^{2}\geq 0,~~\forall t\geq t_{\Vap},
\end{equation}
\begin{equation}
\label{th:genest18}
b_{4}(\Vap)\alpha_{t}-b_{1}\left(\alpha_{t}(t+1)^{-\tau_{3}}+\alpha^{2}_{t}+\alpha_{t}\beta_{t}\right)\geq (b_{4}(\Vap)/2)\alpha_{t},~~\forall t\geq t_{\Vap},
\end{equation}
and the existence of positive constants $b_{5}$ and $\tau_{4}$ such that
\begin{equation}
\label{th:genest19}b_{3}\left(\alpha_{t}(t+1)^{-\tau_{3}}+\alpha^{2}_{t}+\alpha_{t}\beta_{t}\right)\leq b_{5}\alpha_{t}(t+1)^{-\tau_{4}},~~\forall t\geq t_{\Vap}.
\end{equation}
We thus obtain for $t\geq t_{\Vap}$
\begin{equation}
\label{th:genest21}\East\left[V_{t+1}~|~\mathcal{F}_{t}\right]\mathbb{I}\left(\rho_{\Vap}>t\right)\leq\left[\left(1-(b_{4}(\Vap)/2)\alpha_{t}\right)V_{t}+b_{5}\alpha_{t}(t+1)^{-\tau_{4}}\right]\mathbb{I}\left(\rho_{\Vap}>t\right).
\end{equation}
Note that, by definition of $\Gamma_{\Vap}$,
\begin{equation}
\label{th:genest22}\|\wz_{t}\|^{2}\geq\Vap^{2}~~\mbox{on $\{\wz_{t}\in\Gamma_{\Vap}\}$},
\end{equation}
and, hence, by~\eqref{lm:bg14} we conclude that there exists a constant $b_{6}(\Vap)>0$ such that
\begin{equation}
\label{th:genest23}V_{t}\geq b_{6}(\Vap)~~\mbox{on $\{\rho_{\Vap}>t\}$}.
\end{equation}
By~\eqref{th:genest21} we then have for all $t\geq t_{\Vap}$
\begin{equation}
\label{th:genest24}\East\left[V_{t+1}~|~\mathcal{F}_{t}\right]\mathbb{I}\left(\rho_{\Vap}>t\right)\leq\left[V_{t}-b_{7}(\Vap)\alpha_{t}+b_{5}\alpha_{t}(t+1)^{-\tau_{4}}\right]\mathbb{I}\left(\rho_{\Vap}>t\right)
\end{equation}
with $b_{7}(\Vap)$ being another positive constant. Finally, the observation that $\left(b_{7}(\Vap)/2\right)\alpha_{t}\geq b_{5}\alpha_{t}(t+1)^{-\tau_{4}}$ eventually leads to
\begin{align}
\label{th:genest25}\East\left[V_{t+1}~|~\mathcal{F}_{t}\right]\mathbb{I}\left(\rho_{\Vap}>t\right)\leq\left[V_{t}-\left(b_{7}(\Vap)/2\right)\alpha_{t}\right]\mathbb{I}\left(\rho_{\Vap}>t\right)\\=V_{t}\mathbb{I}\left(\rho_{\Vap}>t\right)-b_{8}(\Vap)\alpha_{t}\mathbb{I}\left(\rho_{\Vap}>t\right)
\end{align}
for all $t\geq t_{\Vap}$ (making $t_{\Vap}$ larger if necessary), where $b_{8}(\Vap)=b_{7}(\Vap)/2$.

By~\eqref{th:genest13} we then obtain
\begin{align}
\label{th:genest26}\East\left[V^{\Vap}_{t+1}~|~\mathcal{F}_{t}\right]\leq V_{t}\mathbb{I}\left(\rho_{\Vap}>t\right)+V_{t_{\Vap}}\mathbb{I}\left(\rho_{\Vap}\leq t\right)-b_{8}(\Vap)\alpha_{t}\mathbb{I}\left(\rho_{\Vap}>t\right)\\
= V^{\Vap}_{t}-b_{8}(\Vap)\alpha_{t}\mathbb{I}\left(\rho_{\Vap}>t\right)
\end{align}
for all $t\geq t_{\Vap}$. Note that the $\{\mathcal{F}_{t}\}$-adapted process $\{V^{\Vap}_{t}\}_{t\geq t_{\Vap}}$ satisfies $\East[V^{\Vap}_{t+1}|\mathcal{F}_{t}]\leq V^{\Vap}_{t}$ for all $t\geq t_{\Vap}$; hence, being a (non-negative) supermartingale it converges, i.e., there exists a finite random variable $V_{\Vap}^{\ast}$ such that $V^{\Vap}_{t}\rightarrow V^{\ast}_{\Vap}$ a.s. as $t\rightarrow\infty$. Now consider the $\{\mathcal{F}_{t}\}$-adapted $\mathbb{R}_{+}$-valued process $\{W^{\Vap}_{t}\}$ given by
\begin{equation}
\label{th:genest27} W_{t}^{\Vap}=V_{t}^{\Vap}+b_{8}(\Vap)\sum_{s=0}^{t-1}\alpha_{s}\mathbb{I}\left(\rho_{\Vap}>s\right),
\end{equation}
and note that, by~\eqref{th:genest26} we obtain
\begin{equation}
\label{th:genest28}\East\left[W^{\Vap}_{t+1}~|~\mathcal{F}_{t}\right]\leq V^{\Vap}_{t}-b_{8}(\Vap)\alpha_{t}\mathbb{I}\left(\rho_{\Vap}>t\right)+b_{8}(\Vap)\sum_{s=0}^{t}\alpha_{s}\mathbb{I}\left(\rho_{\Vap}>s\right)=W^{\Vap}_{t}
\end{equation}
for all $t\geq t_{\Vap}$; hence $\{W^{\Vap}_{t}\}_{t\geq t_{\Vap}}$ is a non-negative supermartingale and there exists a finite random variable $W_{\Vap}^{\ast}$ such that $W^{\Vap}_{t}\rightarrow W^{\ast}_{\Vap}$ a.s. as $t\rightarrow\infty$. We then conclude by~\eqref{th:genest27} that the following limit exists:
\begin{equation}
\label{th:genest29}\lim_{t\rightarrow\infty}b_{8}(\Vap)\sum_{s=0}^{t-1}\alpha_{s}\mathbb{I}\left(\rho_{\Vap}>s\right)=W^{\ast}_{\Vap}-V^{\ast}_{\Vap}<\infty~~\mbox{a.s.}
\end{equation}
Given that $\sum_{s=0}^{t-1}\alpha_{s}\rightarrow\infty$ as $t\rightarrow\infty$, the limit condition in~\eqref{th:genest29} is fulfilled only if the summation terminates at a finite time a.s., i.e., we must have $\rho_{\Vap}<\infty$ a.s.

To summarize, we have for each $\Vap\in (0,1)$, $\rho_{\Vap}<\infty$ a.s., i.e., the process $\{\wz_{t}\}$ exits the set $\Gamma_{\Vap}$ in finite time a.s. In particular, for each positive integer $r>1$, let $\rho_{1/r}$ be the stopping time obtained by choosing $\Vap=1/r$ and consider the sequence $\{\wz_{\rho_{1/r}}\}$ (which is well defined due to the a.s. finiteness of each $\rho_{1/r}$) and note that, by definition,
\begin{equation}
\label{th:genest30}\left\|\wz_{\rho_{1/r}}\right\|\in [0,1/r)\cup (r,\infty)~~\mbox{a.s.}
\end{equation}
However, the a.s. boundedness of the sequence $\{\wz_{t}\}$ (see Lemma~\ref{lm:bg}) implies that
\begin{equation}
\label{th:genest31}\Past\left(\left\|\wz_{\rho_{1/r}}\right\|>r~~\mbox{i.o.}\right)=0,
\end{equation}
where i.o. stands for infinitely often as $r\rightarrow\infty$. Hence, by~\eqref{th:genest30} we conclude that there exists a finite random integer valued random variable $r^{\ast}$ such that $\|\wz_{\rho_{1/r}}\|<1/r$ for all $r\geq r^{\ast}$. This, in turn implies that $\|\wz_{\rho_{1/r}}\|\rightarrow 0$ as $r\rightarrow\infty$ a.s., and, in particular, we obtain
\begin{equation}
\label{th:genest32}\Past\left(\liminf_{t\rightarrow\infty}\left\|\wz_{t}\right\|=0\right)=1.
\end{equation}
By~\eqref{lm:bg14} we may also conclude that $\liminf_{t\rightarrow\infty}V_{t}=0$ a.s. Noting that the limit of $\{V_{t}\}$ exists a.s. (see Corollary~\ref{corr:bg}) we further obtain $V_{t}\rightarrow 0$ as $t\rightarrow\infty$ a.s., from which, by another application of~\eqref{lm:bg14}, we conclude that $\wz_{t}\rightarrow 0$ as $t\rightarrow\infty$ a.s. and the desired consistency assertion follows.
\end{proof}

The other major result of this section concerns the pathwise convergence rate of the processes $\{\mathbf{z}_{n}(t)\}$ to $\abtheta$, stated as follows:
\begin{theorem}
\label{th:genrate} Let the processes $\{\mathbf{z}_{n}(t)\}$ be defined as in~\eqref{th:genest1} and the assumptions and hypotheses of Theorem~\ref{th:genest} hold. Then, there exists a constant $\mu>0$ such that for all $n$ we have
\begin{equation}
\label{th:genrate1}\Past\left(\lim_{t\rightarrow\infty}(t+1)^{\mu}\left\|\mathbf{z}_{n}(t)-\abtheta\right\|=0\right)=1.
\end{equation}
\end{theorem}
\begin{remark}
\label{rem:th:genrate} Note that Theorem~\ref{th:genrate} essentially states that $\|\mathbf{z}_{n}(t)-\abtheta\|=o(t^{-\mu})$ a.s. for each agent $n$, and thus provides a pathwise convergence rate guarantee for generic estimators (with time-varying local innovation gains) of the form given in~\eqref{th:genest1}. In particular, we will use Theorem~\ref{th:genrate} to obtain strong consistency of both the auxiliary and refined estimate sequences introduced in~\eqref{aux:1} and~\eqref{opt:1}-\eqref{gain2}, which are in fact instances of the generic estimator process~\eqref{th:genest1}.
\end{remark}

In order to obtain Theorem~\ref{th:genrate}, we will first quantify the rate of agreement among the individual agent estimates. Specifically, we have the following (see Appendix~\ref{sec:app1} for a proof):
\begin{lemma}
\label{lm:consrate} Let the hypotheses of Lemma~\ref{lm:bg} hold. Then, for each pair of agents $n$ and $l$, we have
\begin{equation}
\label{lm:consrate1}\Past\left(\lim_{t\rightarrow\infty}(t+1)^{\tau}\|\mathbf{z}_{n}(t)-\mathbf{z}_{l}(t)\|=0\right)=1,
\end{equation}
for all $\tau\in (0,1-\tau_{2})$.
\end{lemma}

We now complete the proof of Theorem~\ref{th:genrate}.
\begin{proof}[Proof of Theorem~\ref{th:genrate}]
In what follows we reuse the notation and intermediate processes constructed in the proofs of Lemma~\ref{lm:bg} and Theorem~\ref{th:genest}. Recall $\{V_{t}\}$ to be the $\{\mathcal{F}_{t}\}$-adapted process defined in~\eqref{lm:bg5}. By~\eqref{th:genest15} (and the development preceding it) we note that there exist positive constants $b_{1}$, $b_{2}$, and $b_{3}$, and a deterministic time $t_{1}$ (large enough), such that,
\begin{align}
\label{th:genrate15}\East\left[V_{t+1}~|~\mathcal{F}_{t}\right]\leq\left(1+b_{1}\left(\alpha_{t}(t+1)^{-\tau_{3}}+\alpha^{2}_{t}+\alpha_{t}\beta_{t}\right)\right)V_{t}-2\alpha_{t}\mathcal{H}_{t}(\mathbf{z}_{t})\\
-b_{2}\left(\beta_{t}-\beta_{t}^{2}\right)\left\|(\wz_{t})_{\PC}\right\|^{2}+b_{3}\left(\alpha_{t}(t+1)^{-\tau_{3}}+\alpha^{2}_{t}+\alpha_{t}\beta_{t}\right)
\end{align}
for all $t\geq t_{1}$, where the function $\mathcal{H}_{t}(\cdot)$ is defined in~\eqref{prop:Lg3}. By~\eqref{prop:Lg6} we obtain
\begin{align}
\label{th:genrate16}\mathcal{H}_{t}(\mathbf{z})\geq \left(\mathbf{z}_{\C}-\abbtheta\right)^{\top}\left(\bh(\mathbf{z}_{\C})-\bh(\abbtheta)\right)+\left(\mathbf{z}_{\PC}\right)^{\top}\left(\bh(\mathbf{z})-\bh(\abbtheta)\right)\\+\left(\mathbf{z}_{\C}-\abbtheta\right)^{\top}\left(\bh(\mathbf{z})-\bh(\mathbf{z}_{\C})\right)\\
=\left(\mathbf{z}^{a}-\abtheta\right)^{\top}\left(h(\mathbf{z}^{a})-h(\abtheta)\right)+\left(\mathbf{z}_{\PC}\right)^{\top}\left(\bh(\mathbf{z})-\bh(\abbtheta)\right)\\+\left(\mathbf{z}_{\C}-\abbtheta\right)^{\top}\left(\bh(\mathbf{z})-\bh(\mathbf{z}_{\C})\right)
\end{align}
for all $\mathbf{z}\in\mathbb{R}^{NM}$.

Note that, by Proposition~\ref{prop:analytic}, $h(\cdot)$ is continuously differentiable with positive definite gradient $\nabla_{\btheta}h(\abtheta)=I(\abtheta)$ at $\abtheta$; hence, by the mean-value theorem, there exists $\Vap_{0}>0$ such that for all $\btheta\in\mathbb{B}_{\Vap_{0}}(\abtheta)$ we have
\begin{equation}
\label{th:genrate17} h(\btheta)-h(\abtheta)=\left(I(\abtheta)+R(\btheta,\abtheta)\right)\left(\btheta-\abtheta\right),
\end{equation}
where $R(\cdot,\abtheta)$ is a measurable $\mathbb{R}^{M\times M}$-valued function of $\btheta$ such that
\begin{equation}
\label{th:genrate18}\left\|R(\btheta,\abtheta)\right\|\leq \frac{\lambda_{1}(I(\abtheta))}{2}~~~\forall~\btheta\in\mathbb{B}_{\Vap_{0}}(\abtheta),
\end{equation}
with $\lambda_{1}(I(\abtheta))>0$ denoting the smallest eigenvalue of $I(\abtheta)$. Also, observing that the function $\bh(\cdot)$ is locally Lipschitz, we may conclude that there exists a constant $\ell_{\Vap_{0}}$ such that
\begin{equation}
\label{th:genrate1000}
\left\|\bh(\mathbf{z})-\bh(\pz)\right\|\leq\ell_{\Vap_{0}}\left\|\mathbf{z}-\pz\right\|~~~\forall~\mathbf{z},\pz\in\mathbb{B}_{\Vap_{0}}(\abbtheta).
\end{equation}
Now note that, by Theorem~\ref{th:genest}, $\wz_{t}\rightarrow\mathbf{0}$ a.s. as $t\rightarrow\infty$, and, by Lemma~\ref{lm:consrate}, there exists a constant $\tau>0$ such that
\begin{equation}
\label{th:genrate19}\Past\left(\lim_{t\rightarrow\infty}(t+1)^{\tau}\left\|\left(\mathbf{z}_{t}\right)_{\PC}\right\|=0\right)=1.
\end{equation}
Now consider $\delta>0$ (arbitrarily small) and note that by Egorov's theorem there exists a (deterministic) time $t_{\delta}>0$ (chosen to be larger than $t_{1}$ in~\eqref{th:genrate15}), such that $\Past\left(\mathcal{A}_{\delta}\right)\geq 1-\delta$, where $\mathcal{A}_{\delta}$ denotes the event
\begin{equation}
\label{th:genrate20}
\mathcal{A}_{\delta}=\left\{\sup_{t\geq t_{\delta}}\left\|\mathbf{z}_{t}-\abbtheta\right\|\leq\Vap_{0}\right\}\bigcup\left\{\sup_{t\geq t_{\delta}}(t+1)^{\tau}\left\|\left(\mathbf{z}_{t}\right)_{\PC}\right\|\leq\Vap_{0}\right\}.
\end{equation}
Consequently, denoting by $\rho_{\delta}$ the $\{\mathcal{F}_{t}\}$ stopping time
\begin{equation}
\label{th:genrate21}
\rho_{\delta}=\inf\left\{t\geq t_{\delta}~:~\mbox{$\left\|\mathbf{z}_{t}-\abbtheta\right\|>\Vap_{0}$ or $(t+1)^{\tau}\left\|\left(\mathbf{z}_{t}\right)_{\PC}\right\|>\Vap_{0}$}\right\},
\end{equation}
we have that
\begin{equation}
\label{th:genrate22}\Past\left(\rho_{\delta}=\infty\right)\geq 1-\delta.
\end{equation}
Now consider $t\in [t_{\delta},\rho_{\delta})$; noting that
\begin{equation}
\label{th:genrate23}
\|\mathbf{z}^{a}_{t}-\abtheta\|\leq\|\mathbf{z}_{t}-\abbtheta\|\leq \Vap_{0},
\end{equation}
we have by the construction in~\eqref{th:genrate17}-\eqref{th:genrate18}
\begin{align}
\label{th:genrate24}
&
\left(\left(\mathbf{z}_{t}\right)_{\C}-\abbtheta\right)^{\top}\left(\bh\left(\left(\mathbf{z}_{t}\right)_{\C}\right) -\bh\left(\abbtheta\right)\right)=\left(\mathbf{z}^{a}_{t}-\abtheta\right)^{\top}\left(h(\mathbf{z}^{a}_{t}-h(\abtheta)\right)\\
&\geq\left(\|I(\abtheta)\|-\|R(\mathbf{z}^{a}_{t},\abtheta)\|\right)\left\|\mathbf{z}^{a}_{t}-\abtheta\right\|^{2}\geq (1/2)\lambda_{1}(I(\abtheta))\left\|\mathbf{z}^{a}_{t}-\abtheta\right\|^{2}\\
&\geq b_{4}V_{t}
\end{align}
for some constant $b_{4}>0$.

Similarly, using~\eqref{th:genrate1000}, we have the following inequalities for $t\in [t_{\delta},\rho_{\delta})$:
\begin{align}
\label{th:genrate25}
\left(\mathbf{z}_{t}\right)_{\PC}^{\top}\left(\bh(\mathbf{z}_{t})-\bh(\abbtheta)\right)\leq\left\|\left(\mathbf{z}_{t}\right)_{\PC}\right\|\left\|\bh(\mathbf{z}_{t})-\bh(\abbtheta)\right\|\\
\leq\Vap_{0}(t+1)^{-\tau}\ell_{\Vap_{0}}\|\mathbf{z}_{t}-\abbtheta\|\leq\Vap_{0}^{2}(t+1)^{-\tau}\ell_{\Vap_{0}},
\end{align}
and
\begin{align}
\label{th:genrate26}
\left(\left(\mathbf{z}_{t}\right)_{\C}-\abbtheta\right)^{\top}\left(\bh\left(\mathbf{z}_{t}\right)-\bh\left(\left(\mathbf{z}_{t}\right)_{\C}\right)\right)\leq\left\|\left(\mathbf{z}_{t}\right)_{\C}-\abbtheta\right\|.\ell_{\Vap_{0}}\left\|\mathbf{z}_{t}-\left(\mathbf{z}_{t}\right)_{\C}\right\|\\
\leq\Vap_{0}\ell_{\Vap_{0}}\left\|\left(\mathbf{z}_{t}\right)_{\PC}\right\|\leq\Vap_{0}^{2}\ell_{\Vap_{0}}(t+1)^{-\tau}.
\end{align}
Hence, from~\eqref{th:genrate16} and~\eqref{th:genrate24}-\eqref{th:genrate26}, we conclude that for $t\in [t_{\delta},\rho_{\delta})$ we have
\begin{equation}
\label{th:genrate27}
\mathcal{H}_{t}(\mathbf{z}_{t})\geq b_{4}V_{t}-2\Vap_{0}^{2}\ell_{\Vap_{0}}(t+1)^{-\tau}.
\end{equation}
Let $\{V^{\delta}_{t}\}$ be the $\mathbb{R}_{+}$-valued $\{\mathcal{F}_{t}\}$-adapted process such that $V^{\delta}_{t}=V_{t}\mathbb{I}(t<\rho_{\delta})$ for all $t$. Noting that
\begin{equation}
\label{th:genrate28} V^{\delta}_{t+1}=V_{t+1}\mathbb{I}(t+1<\rho_{\delta})\leq V_{t+1}\mathbb{I}(t<\rho_{\delta}),
\end{equation}
we have
\begin{equation}
\label{th:genrate29}\East\left[V^{\delta}_{t+1}~|~\mathcal{F}_{t}\right]\leq\mathbb{I}(t<\rho_{\delta})\East\left[V_{t+1}~|~\mathcal{F}_{t}\right]~~~\forall t.
\end{equation}
For $t\geq t_{\delta}$ we have by~\eqref{th:genrate27}
\begin{equation}
\label{th:genrate30}\mathcal{H}_{t}(\mathbf{z}_{t})\mathbb{I}(t<\rho_{\delta})\geq \left(b_{4}V_{t}-2\Vap_{0}^{2}\ell_{\Vap_{0}}(t+1)^{-\tau}\right)\mathbb{I}(t<\rho_{\delta}),
\end{equation}
hence, it follows from~\eqref{th:genrate15} and~\eqref{th:genrate29} that
\begin{align}
\label{th:genrate31}
\East\left[V^{\delta}_{t+1}~|~\mathcal{F}_{t}\right]\leq\left(1+b_{1}\left(\alpha_{t}(t+1)^{-\tau_{3}}+\alpha^{2}_{t}+\alpha_{t}\beta_{t}\right)\right)V_{t}\mathbb{I}(t<\rho_{\delta})\\-2\alpha_{t}\left(b_{4}V_{t}-2\Vap_{0}^{2}\ell_{\Vap_{0}}(t+1)^{-\tau}\right)\mathbb{I}(t<\rho_{\delta})\\
-b_{2}\left(\beta_{t}-\beta_{t}^{2}\right)\left\|(\wz_{t})_{\PC}\right\|^{2}+b_{3}\left(\alpha_{t}(t+1)^{-\tau_{3}}+\alpha^{2}_{t}+\alpha_{t}\beta_{t}\right)\\
\leq\left(1-\alpha_{t}\left(2b_{4}-b_{1}(t+1)^{-\tau_{3}}-b_{1}\alpha_{t}-b_{1}\beta_{t}\right)\right)V_{t}^{\delta}-b_{2}\left(\beta_{t}-\beta_{t}^{2}\right)\left\|(\wz_{t})_{\PC}\right\|^{2}\\+\alpha_{t}\left(b_{3}(t+1)^{-\tau_{3}}+b_{3}\alpha_{t}+b_{3}\beta_{t}+4\Vap_{0}^{2}\ell_{\Vap_{0}}(t+1)^{-\tau}\right)
\end{align}
for all $t\geq t_{\delta}$. Observing the decay rates of the various coefficients (see~\eqref{weight}), we conclude that there exist a deterministic time $t^{\prime}_{\delta}\geq t_{\delta}$, and positive constants (independent of $\delta$) $b_{5}$, $b_{6}$ and $\tau_{4}$ such that
\begin{equation}
\label{th:genrate32}
\East\left[V^{\delta}_{t+1}~|~\mathcal{F}_{t}\right]\leq\left(1-b_{5}\alpha_{t}\right)V_{t}^{\delta}+b_{6}\alpha_{t}(t+1)^{-\tau_{4}}
\end{equation}
and $b_{5}\alpha_{t}<1$, for all $t\geq t^{\prime}_{\delta}$.

Let us now choose a constant $\omu$ (independently of $\delta$) such that $\omu\in\left(0,b_{5}\wedge\tau_{4}\wedge 1\right)$. Then, using the inequality
\begin{equation}
\label{th:genrate33} (t+1)^{\omu}-t^{\omu}\leq \omu t^{\omu-1}
\end{equation}
we have for all $t> t_{\delta}$
\begin{equation}
\label{th:genrate34}
(t+1)^{\omu}\left(1-b_{5}\alpha_{t-1}\right)\leq t^{\omu}\left(1+\omu.t^{-1}\right)\left(1-b_{5}.t^{-1}\right)\leq t^{\omu}\left(1-(b_{5}-\omu).t^{-1}\right)
\end{equation}
and
\begin{equation}
\label{th:genrate35}
(t+1)^{\omu}t^{-1-\tau_{4}}=\left(1+t^{-1}\right)^{\omu}t^{-1-(\tau_{4}-\omu)}\leq\left(1+(t^{\prime}_{\delta})^{-1}\right)^{\omu}t^{-1-(\tau_{4}-\omu)}.
\end{equation}
Thus, from~\eqref{th:genrate32} we obtain
\begin{align}
\label{th:genrate36} \East\left[(t+1)^{\omu}V^{\delta}_{t}~|~\mathcal{F}_{t-1}\right]\leq\left(1-(b_{5}-\omu).t^{-1}\right)t^{\omu}V_{t-1}^{\delta}+b_{6}\left(1+(t^{\prime}_{\delta})^{-1}\right)^{\omu}t^{-1-(\tau_{4}-\omu)}\\
\leq t^{\omu}V_{t-1}^{\delta}+b^{\delta}_{7}t^{-1-(\tau_{4}-\omu)}
\end{align}
for all $t>t^{\prime}_{\delta}$ for some constant $b^{\delta}_{7}>0$. Since $\tau_{4}>\omu$, we have $\sum_{t}t^{-1-(\tau_{4}-\omu)}<\infty$; denoting by $\{W^{\delta}_{t}\}$ the non-negative $\{\mathcal{F}_{t}\}$-adapted process such that
\begin{equation}
\label{th:genrate37}W^{\delta}_{t}=(t+1)^{\omu}V^{\delta}_{t}+b_{6}\left(1+(t^{\prime}_{\delta})^{-1}\right)^{\omu}\sum_{s=t+1}^{\infty}s^{-1-(\tau_{4}-\omu)},
\end{equation}
we have that $\East[W^{\delta}_{t}|\mathcal{F}_{t-1}]\leq W^{\delta}_{t-1}$ for all $t>t^{\prime}_{\delta}$. Hence, the process $\{W^{\delta}_{t}\}_{t\geq t^{\prime}_{\delta}}$ is a non-negative supermartingale and converges a.s. to a finite non-negative random variable $W^{\delta}_{\ast}$. By~\eqref{th:genrate37} we further conclude that $(t+1)^{\omu}V^{\delta}_{t}\rightarrow W^{\delta}_{\ast}$ a.s. as $t\rightarrow\infty$. Now let $\mu\in (0,\omu)$ be another constant (chosen independently of $\delta$); noting that the limit $W^{\delta}_{\ast}$ is finite, the above convergence leads to
\begin{equation}
\label{th:genrate38}\Past\left(\lim_{t\rightarrow\infty}(t+1)^{\mu}V^{\delta}_{t}=0\right)=1.
\end{equation}
By~\eqref{th:genrate38} and the fact that $V^{\delta}_{t}=V_{t}\mathbb{I}(t<\rho_{\delta})$ for all $t$, we conclude that,
\begin{equation}
\label{th:genrate39}\lim_{t\rightarrow\infty}(t+1)^{\mu}V_{t}=0~\mbox{a.s. on $\{\rho_{\delta}=\infty\}$}.
\end{equation}
Hence, by~\eqref{th:genrate22} we obtain
\begin{equation}
\label{th:genrate40}\Past\left(\lim_{t\rightarrow\infty}(t+1)^{\mu}V_{t}=0\right)\geq 1-\delta.
\end{equation}
Since $\delta>0$ was chosen arbitrarily and $\mu>0$ is independent of $\delta$, we have, in fact, $(t+1)^{\mu}V_{t}\rightarrow 0$ a.s. as $t\rightarrow\infty$ by taking $\delta$ to zero. The desired assertion follows immediately by noting the correspondence between the processes $\{\wz_{t}\}$ and $\{V_{t}\}$ (see~\eqref{lm:bg15}).
\end{proof}

The assertions of Theorem~\ref{th:genrate} may readily be extended to the case of non-uniform (over sample paths) convergence of the matrix gain sequences $\{K_{n}(t)\}$ to their designated limit $\mathcal{K}$ as follows (see Appendix~\ref{sec:app1} for a proof):
\begin{corollary}
\label{corr:genrate} Let the sequences $\{\mathbf{z}_{n}(t)\}$ be defined as in~\eqref{th:genest1}. Let Assumptions~\ref{ass:globobs},\ref{ass:conn}, and~\ref{ass:weight} hold as in the hypotheses of Theorem~\ref{th:genrate} and the matrix gain sequences $\{K_{n}(t)\}$ be such that $(t+1)^{\tau_{3}}\|K_{n}(t)-\mathcal{K}\|\rightarrow 0$ a.s. as $t\rightarrow\infty$ for all $n$. Then, the assertions of Theorem~\ref{th:genrate} continue to hold, i.e., there exists $\mu>0$ such that $(t+1)^{\mu}\|\mathbf{z}_{n}(t)-\abtheta\|\rightarrow 0$ a.s. as $t\rightarrow\infty$ for all $n$.
\end{corollary}

Note that Corollary~\ref{corr:genrate} is in fact a restatement of Theorem~\ref{th:genrate} under the relaxed assumption that the convergence of the matrix gain sequences $\{K_{n}(t)\}$ need not be uniform over sample paths.

\section{Proofs of Main Results}
\label{sec:proof_main_res}

Throughout this section, Assumption~\ref{ass:globobs} and Assumptions~\ref{ass:conn}-\ref{ass:lingrowth} are assumed to hold.

\subsection{Convergence of Auxiliary Estimates and Adaptive Gains}
\label{subsec:gain}

The first result concerns the consistency of the auxiliary estimate sequence $\{\bx_{n}(t)\}$ at each agent. To this end, noting that the evolution of the auxiliary estimates, see~\eqref{aux:1}, corresponds to a specific instantiation of the generic estimator dynamics analyzed in Theorem~\ref{th:genrate} (with $K_{n}(t)=I_{M}$ for all $n$ and $t$), we immediately have the following:
\begin{lemma}
\label{lm:auxcons} For each $n$, the auxiliary estimate sequence $\{\bx_{n}(t)\}$ (see Section~\ref{subsec:alg}) is a strongly consistent estimate of $\abtheta$. In particular, there exists a positive constant $\mu_{0}$ such that $(t+1)^{\mu_{0}}\|\bx_{n}(t)-\abtheta\|\rightarrow 0$ as $t\rightarrow\infty$ a.s. for all $n$.
\end{lemma}

Lemma~\ref{lm:auxcons} and local Lipschitz continuity of the functions $h_{n}(\cdot)$ lead to the following characterization of the adaptive gain sequences $\{K_{n}(t)\}$~\eqref{gain1} driving the local innovation terms of the agent estimates $\{\mathbf{x}_{n}(t)\}$~\eqref{opt:1} (see Appendix~\ref{sec:app2} for a proof):
\begin{lemma}
\label{lm:gainconv} There exists a positive constant $\tau^{\prime}$ such that, for each $n$, the adaptive gain sequence $\{K_{n}(t)\}$, see~\eqref{gain1}, converges a.s. to $N.I^{-1}(\abtheta)$ at rate $\tau^{\prime}$, i.e.,
\begin{equation}
\label{lm:gainconv1}\mathbb{P}\left((t+1)^{\tau^{\prime}}\left\|K_{n}(t)-N.I^{-1}(\abtheta)\right\|=0\right)=1
\end{equation}
where $I(\cdot)$ denotes the centralized Fisher information, see~\eqref{prop:inf2}.
\end{lemma}

As an immediate consequence of the above development, we have the following consistency of the distributed agent estimates $\{\mathbf{x}_{n}(t)\}$:
\begin{corollary}
\label{corr:xconv} For each $n$, the estimate sequence $\{\mathbf{x}_{n}(t)\}$ \textup{(}see Section~\ref{subsec:alg}\textup{)} is a strongly consistent estimate of $\abtheta$, i.e., $\mathbf{x}_{n}(t)\rightarrow\abtheta$ as $\tri$ a.s.
\end{corollary}

\begin{proof} Note that, by Lemma~\ref{lm:gainconv}, there exists $\tau^{\prime}>0$ such that $(t+1)^{\tau^{\prime}}\|K_{n}(t)-N.I^{-1}(\abtheta)\|\rightarrow 0$ as $\tri$ a.s. Thus, the sequences $\{\mathbf{x}_{n}(t)\}$ fall under the purview of Theorem~\ref{th:genrate} (with $\mathcal{K}=N.I^{-1}(\abtheta)$) and the assertion follows.
\end{proof}

\subsection{Proofs of Theorem~\ref{th:estcons} and Theorem~\ref{th:estn}}
\label{subsec:mainproofs} The key idea in proving Theorem~\ref{th:estcons} and Theorem~\ref{th:estn} consists of comparing the nonlinear estimate recursions, see~\eqref{opt:1}, to a suitably linearized recursion. To this end, we consider the following result on distributed linear stochastic recursions developed in~\cite{Kar-AdaptiveDistEst-SICON-2012} in the context of asymptotically efficient distributed parameter estimation in linear multi-agent models. The result to be stated below is somewhat less general than the development in~\cite{Kar-AdaptiveDistEst-SICON-2012}, but serves the current scenario.
\begin{theorem}[Theorem 3.2 and Theorem 3.3 in~\cite{Kar-AdaptiveDistEst-SICON-2012}]
\label{th:optlindist} For each $n$, let $\{\mathbf{v}_{n}(t)\}$ be an $\mathbb{R}^{M}$-valued $\{\mathcal{F}_{t}\}$-adapted process evolving in a distributed fashion as follows:
\begin{equation}
\label{th:optlindist1}
\mathbf{v}_{n}(t+1)=\mathbf{v}_{n}(t)-\beta_{t}\sum_{l\in\Omega_{n}(t)}\left(\mathbf{v}_{n}(t)-\mathbf{v}_{l}(t)\right)+\alpha_{t}D_{n}(t)\left(B_{n}\left(\abtheta-\mathbf{v}_{n}(t)\right)+\bzeta_{n}(t)\right),
\end{equation}
where $B_{n}$, for each $n$, is an $M_{n}\times M$ matrix \textup{(}for some positive integer $M_{n}$\textup{)} such that
\begin{itemize}
\item[(1)] the matrix $A=\sum_{n=1}^{N}B^{\top}_{n}B_{n}$ is positive definite\textup{;}
\item[(2)] for each $n$, the $M\times M_{n}$ matrix valued process $\{D_{n}(t)\}$ is $\{\mathcal{F}_{t}\}$-adapted with $D_{n}(t)\rightarrow N.A^{-1}B^{\top}_{n}$ as $\tri$ a.s.\textup{;}
\item[(3)] for each $n$, the $\{\mathcal{F}_{t+1}$\}-adapted sequence is such that $\{\bzeta_{n}(t)\}$ is independent of $\mathcal{F}_{t}$ for all $t$, the sequence $\{\bzeta_{n}(t)\}$ is i.i.d. with zero mean and covariance $I_{M_{n}}$, and there exists a constant $\Vap>0$ such that $\mathbb{E}[\|\bzeta_{n}(t)\|^{2+\Vap}]<\infty$\textup{;}
\item[(4)] the Laplacian sequence $\{L_{t}\}$ representing the random communication neighborhoods $\Omega_{n}(t)$, $n=1,\cdots,N$, satisfies Assumption~\ref{ass:conn}\textup{;} and
\item[(5)] the weight sequences $\{\alpha_{t}\}$ and $\{\beta_{t}\}$ satisfy Assumption~\ref{ass:weight}.
\end{itemize}
Then the following hold for the processes $\{\mathbf{v}_{n}(t)\}$, $n=1,\cdots,N$\textup{:}
\begin{itemize}
\item[(1)] for each $n$ and $\tau\in [0,1/2)$, we have
\begin{equation}
\label{th:optlindist2}\mathbb{P}\left(\lim_{\tri}(t+1)^{\tau}\left\|\mathbf{v}_{n}(t)-\abtheta\right\|=0\right)=1;
\end{equation}
\item[(2)] for each $n$, the sequence $\{\mathbf{v}_{n}(t)\}$, viewed as an estimate of $\abtheta$, is asymptotically normal with asymptotic covariance $A^{-1}$, i.e.\textup{,}
\begin{equation}
\label{th:optlindist3}
\sqrt{t+1}\left(\mathbf{v}_{n}(t)-\abtheta\right)~\Longrightarrow~\mathcal{N}\left(\mathbf{0},A^{-1}\right).
\end{equation}
\end{itemize}
\end{theorem}

The following corollary to Theorem~\ref{th:optlindist} will be used in the sequel.
\begin{corollary}
\label{corr:linz} For each $n$, let $\{\bv_{n}(t)\}$ be the $\{\mathcal{F}_{t}\}$-adapted $\mathbb{R}^{M}$-valued process evolving in a distributed fashion as
\begin{equation}
\label{corr:linz1}\bv_{n}(t+1)=\bv_{n}(t)-\beta_{t}\sum_{l\in\Omega_{n}(t)}\left(\bv_{n}(t)-\bv_{l}(t)\right)+\alpha_{t}K_{n}(t)\left(I_{n}(\abtheta)\left(\abtheta-\bv_{n}(t)\right)+\mathbf{w}_{n}(t)\right),
\end{equation}
where
\begin{itemize}
\item[(1)] for each $n$, $\{\mathbf{w}_{n}(t)\}$ is the $\{\mathcal{F}_{t+1}\}$-adapted sequence given by $\mathbf{w}_{n}(t)=g_{n}(\mathbf{y}_{n}(t))-h_{n}(\abtheta)$ for all $t$\textup{;}
\item[(2)] for each $n$, $\{K_{n}(t)\}$ denotes the $\{\mathcal{F}_{t}\}$-adapted innovation gain sequence defined as in~\eqref{gain1}\textup{;}
\item[(3)] the Laplacian sequence $\{L_{t}\}$ representing the random communication neighborhoods $\Omega_{n}(t)$, $n=1,\cdots,N$, satisfies Assumption~\ref{ass:conn}\textup{;} and
\item[(4)] the weight sequences $\{\alpha_{t}\}$ and $\{\beta_{t}\}$ satisfy Assumption~\ref{ass:weight}.
\end{itemize}
Then the following hold for the processes $\{\bv_{n}(t)\}$, $n=1,\cdots,N$\textup{:}
\begin{itemize}
\item[(1)] for each $n$ and $\tau\in [0,1/2)$, we have
\begin{equation}
\label{corr:linz3}\Past\left(\lim_{\tri}(t+1)^{\tau}\left\|\bv_{n}(t)-\abtheta\right\|=0\right)=1;
\end{equation}
\item[(2)] for each $n$, the sequence $\{\bv_{n}(t)\}$, viewed as an estimate of $\abtheta$, is asymptotically efficient as per Definition~\ref{def:asyeff}, i.e.\textup{,}
\begin{equation}
\label{corr:linz4}
\sqrt{t+1}\left(\bv_{n}(t)-\abtheta\right)~\Longrightarrow~\mathcal{N}\left(\mathbf{0},I^{-1}(\abtheta)\right).
\end{equation}
\end{itemize}
\end{corollary}

\begin{proof} Note that, by Proposition~\ref{prop:inf}, for each $n$ the local Fisher information matrix $I_{n}(\abtheta)$ is positive semidefinite; hence, there exists (for example, by a Cholesky factorization) a positive integer $M_{n}$ and an $M_{n}\times M$ matrix $B_{n}$ such that $I_{n}(\abtheta)=B_{n}^{\top}B_{n}$.
By Proposition~\ref{prop:analytic}, for each $n$, the sequence $\{\mathbf{w}_{n}(t)\}$ possesses moments of all orders, is zero-mean with covariance $I_{n}(\btheta)$. Since $I_{n}(\abtheta)=B_{n}^{\top}B_{n}$, there exists another $\{\mathcal{F}_{t+1}$ adapted sequence $\{\bzeta_{n}(t)\}$ (not necessarily unique depending on the rank of the matrix $B_{n}$) satisfying condition (3) in the hypothesis of Theorem~\ref{th:optlindist}, such that $B_{n}^{\top}\bzeta_{n}(t)=\mathbf{w}_{n}(t)$ for all $t$ a.s.

Also, for each $n$, denote by $\{D_{n}(t)\}$ the $M\times M_{n}$ matrix-valued $\{\mathcal{F}_{t}\}$-adapted process such that $D_{n}(t)=K_{n}(t)B_{n}^{\top}$ for all $t$; since, by Lemma~\ref{lm:gainconv}, $K_{n}(t)\rightarrow N.I^{-1}(\abtheta)$ as $\tri$ a.s., we have that $D_{n}(t)\rightarrow N.I^{-1}(\abtheta)B_{n}^{\top}$ as $\tri$ a.s.

It is now clear, that the evolution of the sequences $\{\bv_{n}(t)\}$ may be rewritten as follows in terms of the newly introduced variables:
\begin{equation}
\label{corr:linz5} \bv_{n}(t+1)=\bv_{n}(t)-\beta_{t}\sum_{l\in\Omega_{n}(t)}\left(\bv_{n}(t)-\bv_{l}(t)\right)+\alpha_{t}D_{n}(t)\left(B_{n}\left(\abtheta-\bv_{n}(t)\right)+\bzeta_{n}(t)\right).
\end{equation}
Finally noting that, by construction and Proposition~\ref{prop:inf},
\begin{equation}
\label{corr:linz6} I(\abtheta)=\sum_{n=1}^{N}I_{n}(\abtheta)=\sum_{n=1}^{N}B_{n}^{\top}B_{n},
\end{equation}
we conclude that the evolution in~\eqref{corr:linz5} falls under the purview of Theorem~\ref{th:optlindist} (with the identification that $A=I(\abtheta)$) and the desired assertions follow.
\end{proof}

The processes $\{\bv_{n}(t)\}$ as introduced and analyzed in Corollary~\ref{corr:linz} may, in fact, be viewed as linearizations of the nonlinear estimator dynamics, see~\eqref{opt:1}, the linearizations being performed in the vicinity of the true parameter value $\abtheta$. Clearly, in order for such linearization to provide meaningful insight into the actual nonlinear dynamics of the estimators $\{\mathbf{x}_{n}(t)\}$'s, it is necessary that the latter approach stay close to~$\abtheta$ (around which the linearization is carried out) asymptotically, which, in turn, is guaranteed by the consistency of the estimators $\{\mathbf{x}_{n}(t)\}$'s, see Corollary~\ref{corr:xconv}. The consistency allows us to obtain insight into the detailed dynamics of the estimators $\{\mathbf{x}_{n}(t)\}$'s by characterizing the pathwise deviations of the former from their linearized counterparts. These ideas are formalized below in Lemma~\ref{lm:dev} (see Appendix~\ref{sec:app2} for a proof) leading to the main results of the paper as presented in Section~\ref{subsec:mainres}.
\begin{lemma}
\label{lm:dev} For each $n$, let $\{\mathbf{x}_{n}(t)\}$ be the estimate sequence at agent $n$ as defined in~\eqref{opt:1}, and $\{\bv_{n}(t)\}$ denote the process defined in~\eqref{corr:linz1} under the hypotheses of Corollary~\ref{corr:linz}. Then, there exists a constant $\overline{\tau}>1/2$ such that
\begin{equation}
\label{lm:dev1}\Past\left(\lim_{\tri}(t+1)^{\overline{\tau}}\left\|\mathbf{x}_{n}(t)-\bv_{n}(t)\right\|=0\right)=1
\end{equation}
for all $n$.
\end{lemma}

With the above development, we may now complete the proofs of Theorem~\ref{th:estcons} and Theorem~\ref{th:estn} as follows.

\begin{proof}[Proof of Theorem~\ref{th:estcons}] Let $\tau\in [0,1/2)$ and note that, for each $n$,
\begin{equation}
\label{th:estcons2}(t+1)^{\tau}\left\|\mathbf{x}_{n}(t)-\abtheta\right\|\leq (t+1)^{\tau}\left\|\mathbf{x}_{n}(t)-\bv_{n}(t)\right\| + (t+1)^{\tau}\left\|\bv_{n}(t)-\abtheta\right\|,
\end{equation}
where $\{\bv_{n}(t)\}$ is the (\emph{linearized}) approximation introduced and analyzed in Corollary~\ref{corr:linz}. By Lemma~\ref{lm:dev} (first assertion) and Corollary~\ref{corr:linz}, since $\tau<1/2$, we have $(t+1)^{\tau}\|\mathbf{x}_{n}(t)-\bv_{n}(t)\|\rightarrow 0$ and $(t+1)^{\tau}\|\bv_{n}(t)-\abtheta\|\rightarrow 0$ respectively as $\tri$ a.s. Hence, by~\eqref{th:estcons2} we obtain $(t+1)^{\tau}\|\mathbf{x}_{n}(t)-\abtheta\|\rightarrow 0$ as $\tri$ a.s., thus establishing Theorem~\ref{th:estcons}.
\end{proof}

\begin{proof}[Proof of Theorem~\ref{th:estn}]
Note that by Lemma~\ref{lm:dev} (first assertion), for each $n$
\begin{align}
\label{th:estn3}
\Past\left(\lim_{\tri}\left\|\sqrt{t+1}\left(\mathbf{x}_{n}(t)-\abtheta\right)-\sqrt{t+1}\left(\bv_{n}(t)-\abtheta\right)\right\|=0\right)\\=\Past\left(\lim_{\tri}\sqrt{t+1}\left\|\mathbf{x}_{n}(t)-\bv_{n}(t)\right\|=0\right)=1.
\end{align}
Hence, in particular, the sequences $\{\sqrt{t+1}(\mathbf{x}_{n}(t)-\abtheta)\}$ and $\{\sqrt{t+1}(\bv_{n}(t)-\abtheta)\}$ possess the same weak limit (if the latter exists for one of the sequences); the asymptotic normality (efficiency) in Theorem~\ref{th:estn} then follows immediately by the corresponding for the $\{\bv_{n}(t)\}$ sequence in Corollary~\ref{corr:linz} (second assertion).
\end{proof}

\section{Comparisons, Complexity and Trade-Offs}
\label{sec:disc} We summarize some of the key technical constructs in the distributed scheme~\eqref{aux:1}-\eqref{gain2}, discuss possible modifications and simplifications, and complexity-performance trade-offs. First, note that if the goal is to obtain consistent estimates only, the optimal estimate generation~\eqref{opt:1} and adaptive gain refinement~\eqref{gain1}-\eqref{gain2} steps are not required, as the auxiliary estimate sequence $\{\bx_{n}(t)\}$ is consistent for each $n$ (see Lemma~\ref{lm:auxcons}). In fact, a little more work along the lines of Theorems~\ref{th:estcons}-\ref{th:estn} will show that these auxiliary estimates are order optimal as far as consistency is concerned and asymptotically normal, i.e., for each $n$ we have
\begin{equation}
\label{disc1} \lim_{t\rightarrow\infty}(t+1)^{\tau}\|\bx_{n}(t)-\abtheta\| = 0,~~~\forall\tau\in [0,1/2),
\end{equation}
and
\begin{equation}
\label{disc2}
(t+1)^{1/2}\left(\bx_{n}(t)-\mathbf{\theta}^{\ast}\right)\Longrightarrow\mathcal{N}\left(\mathbf{0},\bV_{n}(\abtheta)\right)
\end{equation}
for some positive definite matrix $\bV_{n}(\abtheta)$ which depends on $\abtheta$ only. Clearly, we will have that $\bV_{n}(\abtheta)\succeq I^{-1}(\abtheta)$ for all $\abtheta$ and, in general, there will be instances of $\abtheta$ such that $\bV_{n}(t)\neq I^{-1}(\abtheta)$. Specifically, if the statistical model is such that the Fisher information matrix $I(\abtheta)$ is not a constant function of $\abtheta$, i.e., varies with $\abtheta$, then a constant (innovation) gain estimator, of which~\eqref{aux:1} is a particular instance, may never achieve an asymptotic covariance of $I^{-1}(\abtheta)$ for all realizations of the parameter $\abtheta$. This observation is consistent with the theory of centralized recursive estimation and may also be shown by directly evaluating the asymptotic covariances of the auxiliary estimate sequences.

We now discuss the implications of the design construction $\beta_{t}/\alpha_{t}\rightarrow\infty$ as $t\rightarrow\infty$, which makes our distributed scheme (including the auxiliary estimate generation, the optimal estimate generation and the gain update steps) mixed time-scale, i.e., the consensus potential asymptotically dominates the (local) innovation potential. Such a construction is in fact necessary to achieve asymptotic efficiency. The mixed time-scale construction is one of the key distinctive features of the distributed scheme~\eqref{aux:1}-\eqref{gain2} with respect to prior work on recursive distributed estimation, see~\cite{KarMouraRamanan-Est-2008} for example. Constant innovation gain schemes of the form~\eqref{aux:1} are single time-scale dynamics, i.e., for them\footnote{For two positive sequences $\{f_{t}\}$ and $\{g_{t}\}$, the notation $f_{t}=\Omega(g_{t})$ means that there exists positive numbers $c_{1}$ and $c_{2}$ such that $c_{1}g_{t}\leq f_{t}\leq c_{2}g_{t}$ for all $t$.} $\beta_{t}/\alpha_{t}=\Omega(1)$. Under broad conditions, we showed that these algorithms are consistent and asymptotically normal (see Proposition~16 and Theorems~10, 18 and~19 in~\cite{KarMouraRamanan-Est-2008}); however, the asymptotic covariance  obtained at each agent was shown to be a function of both the communication network topology and the parameter value $\abtheta$. In contrast, note that the asymptotic covariance attained by the distributed scheme~\eqref{aux:1} (and of course, the asymptotically efficient scheme~\eqref{opt:1}) is invariant to the network topology--this invariance is critical to obtaining asymptotic efficiency (achieved by adaptively tuning the local innovation gains as in steps~\eqref{opt:1}-\eqref{gain2}) for arbitrary communication topologies (satisfying Assumption~\ref{ass:conn}). Such asymptotic efficiency is not achievable by the single time-scale schemes in~\cite{KarMouraRamanan-Est-2008} with constant or adaptive local innovation gains. In summary, we note that both the mixed time-scale construction, i.e., the requirement that $\beta_{t}/\alpha_{t}\rightarrow\infty$ as $t\rightarrow\infty$, and the adaptive tuning of the local innovation gains (achieved in this case through the refinement steps~\eqref{opt:1}-\eqref{gain2}) are critical to obtaining asymptotically efficient estimates.

We conclude this section by briefly discussing the implementation complexity of the distributed scheme~\eqref{aux:1}-\eqref{gain2}. Note that, in comparison to the single time-scale schemes proposed in~\cite{KarMouraRamanan-Est-2008}, additional complexity is incurred in the adaptive gain refinement steps which involves computing the local Fisher matrices at each agent and the matrix-valued message passing for updating the $G_{n}(t)$'s, see~\eqref{gain2}. In many applications, an analytical form of the local Fisher matrix function is available, hence, the major complexity is in communicating the matrix-valued $G_{n}(t)$'s for the distributed update. However, the matrix-value message passing may be completely eliminated if global statistical information is available at each agent. Specifically, if each agent $n$ is aware of the Fisher matrix functions of all the other agents, the gain update steps~\eqref{gain1}-\eqref{gain2} may be reduced to completely decentralized computations of the form
\begin{equation}
\label{disc3} K_{n}(t)=\left(\frac{1}{N}\sum_{l=1}^{N}I_{l}(\bx_{n}(t))+\varphi_{t}I_{M}\right)^{-1},~~~\forall n,
\end{equation}
thus completely eliminating the need for additional message passing for updating the adaptive gain sequence. In fact, Theorems~\ref{th:estcons}-\ref{th:estn} will continue to hold if the gain update steps~\eqref{gain1}-\eqref{gain2} are replaced by~\eqref{disc3} as the assertions of Lemma~\ref{lm:gainconv} continues to hold by the consistency of the sequences $\{\bx_{n}(t)\}$ if the $\{K_{n}(t)\}$ are set according to~\eqref{disc3}. In a sense, we observe an interesting trade-off between knowledge and complexity: if the agents possess knowledge of the global statistical model, the commmunication complexity may be reduced further without any loss of performance.

\section{Conclusions}
\label{sec:conclusions}
We have developed distributed estimators of the consensus + innovations type for multi-agent scenarios with general exponential family observation statistics that yield consistent and asymptotically efficient parameter estimates at all agents. Moreover, the above estimator properties and optimality hold as long as the aggregate or global sensing model is observable and the inter-agent communication network is connected in the mean (otherwise, irrespective of the network sparsity). Along the way, we have characterized analogues of classical system and information theoretic notions such as observability to the distributed-information setting.

An important future research question arises naturally: in this paper we have assumed that the parametrization is continuous unconstrained, i.e., $\btheta$ may take values over the entire space $\mathbb{R}^{M}$. It would be of interest to extend the approach to account for constrained parametrization--the parameter $\btheta$ could belong to a restricted subset $\Theta\subset\mathbb{R}^{M}$ either because of direct physical constraints or due to constrained natural parameterizations of the local exponential families involved, i.e., the domains of definition of the functions $\lambda_{n}(\cdot)$ in~\eqref{ass:sensmod1} being strict subsets of $\mathbb{R}^{M}$. A specific instance arises with finite classification or detection (hypothesis testing) problems in which $\btheta$ may only assume a finite set of values\footnote{We are somewhat abusing the notion of estimation which, to be precise, corresponds to inferring \emph{continuous} parameters as pursued in this paper. However, by considering constrained parametrization, we are essentially expanding its usage to general inference problems including detection and classification.}. The unconstrained estimation approach~\eqref{aux:1}-\eqref{gain2} may still be applicable to a subclass of such constrained cases by considering \emph{suitable} analytical extensions of the various functions $\lambda_{n}(\cdot)$'s, $h_{n}(\cdot)$'s etc. over $\mathbb{R}^{M}$; provided such extensions exist\footnote{An idea related to such analytical extensions is that of embedding into an exponential family (see~\cite{Rukhin1994recursive} for some discussion in a related but centralized context), in which, broadly speaking, the objective is to obtain an unconstrained exponential family whose restriction to $\Theta$ coincides with the given constrained family.}, the proposed algorithm will lead to asymptotically efficient estimates at the network agents although the intermediate iterates may not belong to $\Theta$. As a familiar example where such extension may be achievable by embedding, we may envision a binary hypothesis testing problem corresponding to the presence or absence of a signal observed in additive zero-mean Gaussian noise with known variance. In cases, where such analytical extensions may not be obtained, other modifications of the proposed scheme, for example by supplementing the local estimate update processes with a projection step onto the set $\Theta$, may be helpful. In the interest of obtaining a unified distributed inference framework, it would be worthwhile to study such extensions and modifications of the proposed scheme.\\


\appendices

\section{Proofs of Results in Section~\ref{sec:genconsest}}
\label{sec:app1}

\begin{proof}[Proof of Proposition~\ref{prop:Lg}] Let $\mathbf{z}\in\Gamma_{\Vap}$ and note that by reasoning along the lines of~\eqref{lm:bg201} we obtain
\begin{equation}
\label{prop:Lg5}\left(\mathbf{z}-\abbtheta\right)^{\top}\left(\OL\otimes\mathcal{K}^{-1}\right)\left(\mathbf{z}-\abbtheta\right)=\left(\mathbf{z}_{\PC}\right)^{\top}\left(\OL\otimes\mathcal{K}^{-1}\right)\mathbf{z}_{\PC}\geq\lambda_{2}(\OL)\lambda_{1}(\mathcal{K}^{-1})\left\|\mathbf{z}_{\PC}\right\|^{2},
\end{equation}
where $\mathbf{z}_{\PC}$ denotes the projection of $\mathbf{z}$ onto the orthogonal complement of the consensus subspace (see Definition~\ref{def:consspace}).
We thus obtain
\begin{align}
\label{prop:Lg6}\mathcal{H}_{t}(\mathbf{z})\geq\frac{b_{\beta}\beta_{t}}{\alpha_{t}}\lambda_{2}(\OL)\lambda_{1}(\mathcal{K}^{-1})\left\|\mathbf{z}_{\PC}\right\|^{2}+\left(\mathbf{z}-\abbtheta\right)^{\top}\left(\bh(\mathbf{z})-\bh(\abbtheta)\right)\\
\geq \frac{b_{\beta}\beta_{t}}{\alpha_{t}}\lambda_{2}(\OL)\lambda_{1}(\mathcal{K}^{-1})\left\|\mathbf{z}_{\PC}\right\|^{2} + \left(\mathbf{z}_{\C}-\abbtheta\right)^{\top}\left(\bh(\mathbf{z}_{\C})-\bh(\abbtheta)\right)\\+\left(\mathbf{z}_{\PC}\right)^{\top}\left(\bh(\mathbf{z})-\bh(\abbtheta)\right)+\left(\mathbf{z}_{\C}-\abbtheta\right)^{\top}\left(\bh(\mathbf{z})-\bh(\mathbf{z}_{\C})\right).
\end{align}
In order to bound the last two terms in the above inequality, note that, for $\mathbf{z}\in\Gamma_{\Vap}$, by invoking Assumption~\ref{ass:lingrowth} we obtain
\begin{equation}
\label{prop:Lg7}\left|\left(\mathbf{z}_{\PC}\right)^{\top}\left(\bh(\mathbf{z})-\bh(\abbtheta)\right)\right|\leq c_{1}\left\|\mathbf{z}_{\PC}\right\|\left(1+\left\|\mathbf{z}-\abbtheta\right\|\right)\leq c_{1}\left(1/\Vap + 1\right)\left\|\mathbf{z}_{\PC}\right\|,
\end{equation}
where $c_{1}$ is a positive constant. Also, by Proposition~\ref{prop:analytic}, the functions $h_{n}(\cdot)$ are infinitely continuously differentiable and hence locally Lipschitz; in particular, noting that the set $\Gamma_{\Vap}^{\prime}$
\begin{equation}
\label{prop:Lg8}\Gamma^{\prime}_{\Vap}=\left\{\mathbf{z}\in\mathbb{R}^{NM}~:~\left\|\mathbf{z}-\abbtheta\right\|\leq 1/\Vap\right\}
\end{equation}
is compact, there exists a constant $\ell_{\Gamma_{\Vap}}>0$, such that,
\begin{equation}
\label{prop:Lg9}\left\|\bh(\mathbf{z})-\bh(\mathbf{z}^{\prime})\right\|\leq\ell_{\Gamma_{\Vap}}\left\|\mathbf{z}-\mathbf{z}^{\prime}\right\|,~~~\forall\mathbf{z},\mathbf{z}^{\prime}\in\Gamma_{\Vap}^{\prime}.
\end{equation}
By observing that for $\mathbf{z}\in\Gamma_{\Vap}\subset\Gamma_{\Vap}^{\prime}$
\begin{equation}
\label{prop:Lg10}\left\|\mathbf{z}_{\C}-\abbtheta\right\|=\left\|\mathbf{z}-\abbtheta\right\|-\left\|\mathbf{z}_{\PC}\right\|\leq \left\|\mathbf{z}-\abbtheta\right\|\leq 1/\Vap,
\end{equation}
we obtain $\mathbf{z}_{\C}\in\Gamma^{\prime}_{\Vap}$; hence, by~\eqref{prop:Lg9}, we may conclude that
\begin{equation}
\label{prop:Lg11}\left\|\bh(\mathbf{z})-\bh(\mathbf{z}_{\C})\right\|\leq\ell_{\Gamma^{\prime}_{\Vap}}\left\|\mathbf{z}-\mathbf{z}_{\C}\right\|=\ell_{\Gamma^{\prime}_{\Vap}}\left\|\mathbf{z}_{\PC}\right\|
\end{equation}
for all $\mathbf{z}\in\Gamma_{\Vap}$. Thus, for $\mathbf{z}\in\Gamma_{\Vap}$, we have
\begin{equation}
\label{prop:Lg12}\left|\left(\mathbf{z}_{\C}-\abbtheta\right)^{\top}\left(\bh(\mathbf{z})-\bh(\mathbf{z}_{\C})\right)\right|\leq\ell_{\Gamma^{\prime}_{\Vap}}\left\|\mathbf{z}_{\PC}\right\|\left\|\mathbf{z}-\abbtheta\right\|\leq\left(\ell_{\Gamma^{\prime}_{\Vap}}/\Vap\right)\left\|\mathbf{z}_{\PC}\right\|.
\end{equation}
Combining~\eqref{prop:Lg5}-\eqref{prop:Lg7} and~\eqref{prop:Lg12} we then obtain
\begin{align}
\label{prop:Lg13}\mathcal{H}_{t}(\mathbf{z})\geq\left(\frac{b_{\beta}\beta_{t}}{\alpha_{t}}\lambda_{2}(\OL)\lambda_{1}(\mathcal{K}^{-1})\left\|\mathbf{z}_{\PC}\right\|-c_{1}\left(\frac{1}{\Vap}+1\right)-\frac{\ell_{\Gamma^{\prime}_{\Vap}}}{\Vap}\right)\left\|\mathbf{z}_{\PC}\right\|\\
+\left(\mathbf{z}_{\C}-\abbtheta\right)^{\top}\left(\bh(\mathbf{z}_{\C})-\bh(\abbtheta)\right)
\end{align}
for all $\mathbf{z}\in\Gamma_{\Vap}$. By invoking standard properties of quadratic minimization, we note that there exist positive constants $\bar{c}_{\Vap}$, $c_{3}(\Vap)$ and $c_{4}(\Vap)$ such that for all $t\geq 0$
\begin{equation}
\label{prop:Lg14}\left(\frac{b_{\beta}\beta_{t}}{\alpha_{t}}\lambda_{2}(\OL)\lambda_{1}(\mathcal{K}^{-1})\left\|\mathbf{z}_{\PC}\right\|-c_{1}\left(\frac{1}{\Vap}+1\right)-\frac{\ell_{\Gamma^{\prime}_{\Vap}}}{\Vap}\right)\left\|\mathbf{z}_{\PC}\right\|> \bar{c}_{\Vap}
\end{equation}
for all $\mathbf{z}$ with $\left\|\mathbf{z}_{\PC}\right\|> c_{3}(\Vap)\left(\alpha_{t}/\beta_{t}\right)$, and
\begin{equation}
\label{prop:Lg15}\left(\frac{b_{\beta}\beta_{t}}{\alpha_{t}}\lambda_{2}(\OL)\lambda_{1}(\mathcal{K}^{-1})\left\|\mathbf{z}_{\PC}\right\|-c_{1}\left(\frac{1}{\Vap}+1\right)-\frac{\ell_{\Gamma^{\prime}_{\Vap}}}{\Vap}\right)\left\|\mathbf{z}_{\PC}\right\|\geq
-c_{4}(\Vap)\left(\alpha_{t}/\beta_{t}\right)
\end{equation}
for all $\mathbf{z}$, in particular, $\mathbf{z}\in\Gamma_{\Vap}$.

Now, note that, by Proposition~\ref{prop:analytic} (third assertion), for all $\mathbf{z}\in\mathbb{R}^{NM}$ with $\mathbf{z}_{\C}\neq\abbtheta$ we have
\begin{equation}
\label{prop:Lg16}\left(\mathbf{z}_{\C}-\abbtheta\right)^{\top}\left(\bh(\mathbf{z}_{\C})-\bh(\abbtheta)\right)=\sum_{n=1}^{N}\left(\mathbf{z}_{\C}^{a}-\abtheta\right)^{\top}\left(h_{n}(\mathbf{z}_{\C}^{a})-h_{n}(\abtheta)\right)>0
\end{equation}
(see also Definition~\ref{def:consspace}). Let us choose $\Vap^{\prime}$ such that $\Vap^{\prime}\in (0,\Vap)$; noting that the functions $h_{n}(\cdot)$ are continuous and the set $\Phi_{\Vap,\Vap^{\prime}}$
\begin{equation}
\label{prop:Lg17}\Phi_{\Vap,\Vap^{\prime}}=\Gamma_{\Vap}\bigcap\left\{\mathbf{z}\in\mathbb{R}^{NM}~:~\left\|\mathbf{z}_{\C}-\abbtheta\right\|\geq\Vap^{\prime}\right\}
\end{equation}
is compact, we conclude that there exists $\delta_{\Vap}>0$ such that
\begin{equation}
\label{prop:Lg18}\inf_{\mathbf{z}\in\Phi_{\Vap,\Vap^{\prime}}}\left(\mathbf{z}_{\C}-\abbtheta\right)^{\top}\left(\bh(\mathbf{z}_{\C})-\bh(\abbtheta)\right)>\delta.
\end{equation}
Further, since $\alpha_{t}/\beta_{t}\rightarrow 0$ as $t\rightarrow\infty$ (by hypothesis), there exist $t_{\Vap}$ large enough and a constant $\bar{\delta}_{\Vap}>0$ such that
\begin{equation}
\label{prop:Lg19}
\Vap^{\prime}<\Vap-c_{3}(\Vap)\left(\alpha_{t}/\beta_{t}\right)~~~\mbox{and}~~~c_{4}(\Vap)\left(\alpha_{t}/\beta_{t}\right)<\delta_{\Vap}-\bar{\delta}_{\Vap}
\end{equation}
for all $t\geq t_{\Vap}$.

We now show that there exists $\delta_{\Vap}^{\prime}>0$ (independent of $t$) such that
\begin{equation}
\label{prop:Lg20}\inf_{\mathbf{z}\in\Gamma_{\Vap}}\mathcal{H}_{t}(\mathbf{z})>\delta^{\prime}_{\Vap}
\end{equation}
for all $t\geq t_{\Vap}$. To this end, for $t\geq t_{\Vap}$, let $\mathbf{z}\in\Gamma_{\Vap}$ and consider the two cases as to whether $\left\|\mathbf{z}_{\PC}\right\|> c_{3}(\Vap)\left(\alpha_{t}/\beta_{t}\right)$ or not. Noting that by Proposition~\ref{prop:analytic}
\begin{equation}
\label{prop:Lg21}\left(\mathbf{z}_{\C}-\abbtheta\right)^{\top}\left(\bh(\mathbf{z}_{\C})-\bh(\abbtheta)\right)\geq 0,~~~\forall\mathbf{z}\in\mathbb{R}^{NM},
\end{equation}
we have by~\eqref{prop:Lg13}-\eqref{prop:Lg14} that
\begin{equation}
\label{prop:Lg22}
\mathcal{H}_{t}(\mathbf{z})>\bar{c}_{\Vap}
\end{equation}
for all $\mathbf{z}\in\Gamma_{\Vap}$ with $\left\|\mathbf{z}_{\PC}\right\|> c_{3}(\Vap)\left(\alpha_{t}/\beta_{t}\right)$. Now consider the other case, i.e., let $\mathbf{z}\in\Gamma_{\Vap}$ with $\left\|\mathbf{z}_{\PC}\right\|\leq c_{3}(\Vap)\left(\alpha_{t}/\beta_{t}\right)$; note that, since $t\geq t_{\Vap}$, we have  for such $\mathbf{z}$ by~\eqref{prop:Lg19}
\begin{equation}
\label{prop:Lg23}\left\|\mathbf{z}_{\C}-\abbtheta\right\|=\left\|\mathbf{z}-\abbtheta\right\|-\left\|\mathbf{z}_{\PC}\right\|\geq\Vap-c_{3}(\Vap)\left(\alpha_{t}/\beta_{t}\right)\Vap^{\prime}.
\end{equation}
Hence, necessarily $\mathbf{z}\in\Phi_{\Vap,\Vap^{\prime}}$ and we have by~\eqref{prop:Lg13},\eqref{prop:Lg15}, and~\eqref{prop:Lg18}-\eqref{prop:Lg19}
\begin{equation}
\label{prop:Lg24}\mathcal{H}_{t}(\mathbf{z})>\delta_{\Vap}-c_{4}(\Vap)\left(\alpha_{t}/\beta_{t}\right)>\bar{\delta}_{\Vap}.
\end{equation}
From~\eqref{prop:Lg22} and~\eqref{prop:Lg24} we then obtain for all $\mathbf{z}\in\Gamma_{\Vap}$ and $t\geq t_{\Vap}$
\begin{equation}
\label{prop:Lg25}
\mathcal{H}_{t}(\mathbf{z})>\delta_{\Vap}^{\prime}>0,
\end{equation}
where $\delta_{\Vap}^{\prime}=\bar{c}_{\Vap}\wedge\bar{\delta}_{\Vap}$, thus establishing the assertion in~\eqref{prop:Lg20}.

Finally, let $\bc_{\Vap}=\Vap^{2}\bar{\delta}_{\Vap}$, and note that the desired claim follows by
\begin{equation}
\label{prop:Lg26}\mathcal{H}_{t}(\mathbf{z})>\delta_{\Vap}^{\prime}\geq\delta_{\Vap}^{\prime}\left(\Vap^{2}\left\|\mathbf{z}-\abbtheta\right\|^{2}\right)=\bc_{\Vap}\left\|\mathbf{z}-\abbtheta\right\|^{2}
\end{equation}
for all $\mathbf{z}\in\Gamma_{\Vap}$ and $t\geq t_{\Vap}$, where we used that fact that $\Vap^{2}\left\|\mathbf{z}-\abbtheta\right\|^{2}\leq 1$ on $\Gamma_{\Vap}$.
\end{proof}

{\bf Proof of Lemma~\ref{lm:consrate}}: Before proceeding to the proof of Lemma~\ref{lm:consrate}, we state the following approximation results obtained in~\cite{Kar-AdaptiveDistEst-SICON-2012} on convergence estimates of stochastic recursions (Lemma~\ref{lm:mean-conv}) and certain attributes of time-varying stochastic Laplacian matrices (Lemma~\ref{lm:conn}), to be used as intermediate ingredients in the proof.

\begin{lemma}[Lemma 4.3 in~\cite{Kar-AdaptiveDistEst-SICON-2012}]
\label{lm:mean-conv} Let $\{\mathbf{w}_{t}\}$ be an $\mathbb{R}_{+}$-valued $\{\mathcal{F}_{t}\}$ adapted process that satisfies
\begin{equation}
\label{lm:mean-conv1}
\mathbf{w}_{t+1}\leq \left(1-r_{1}(t)\right)\mathbf{w}_{t}+r_{2}(t)U_{t}\left(1+J_{t}\right).
\end{equation}
In the above, $\{r_{1}(t)\}$ is an $\{\mathcal{F}_{t+1}\}$ adapted process, such that for all $t$, $r_{1}(t)$ satisfies $0\leq r_{1}(t)\leq 1$ and
\begin{equation}
\label{lm:JSTSP2}
\frac{a_{1}}{(t+1)^{\delta_{1}}}\leq\mathbb{E}\left[r_{1}(t)~|~\mathcal{F}_{t}\right]\leq 1
\end{equation}
with $a_{1}>0$ and $0\leq \delta_{1}< 1$; $\{r_{2}(t)\}$ is a deterministic sequence satisfying $r_{2}(t)\leq a_{2}.(t+1)^{-\delta_{2}}$ for all $t\geq 0$, where $a_{2}$ and $\delta_{2}$ are positive constants.
Further, let $\{U_{t}\}$ and $\{J_{t}\}$ be $\mathbb{R}_{+}$ valued $\{\mathcal{F}_{t}\}$ and $\{\mathcal{F}_{t+1}\}$ adapted processes respectively with $\sup_{t\geq 0}U_{t}<\infty$ a.s. The process $\{J_{t}\}$ is i.i.d.~with $J_{t}$ independent of $\mathcal{F}_{t}$ for each $t$ and satisfies the moment condition $\mathbb{E}\left[J_{t}^{2+\varepsilon}\right]<\infty$ for a constant $\varepsilon>0$. Then, if $\delta_{2}>\delta_{1}+1/(2+\varepsilon)$,
we have $(t+1)^{\delta_{0}}\mathbf{w}_{t}\rightarrow 0$ a.s. as $t\rightarrow\infty$ for all $\delta_{0}\in [0,\delta_{2}-\delta_{1}-1/(2+\varepsilon))$.
\end{lemma}

\begin{lemma}[Lemma 4.4 in~\cite{Kar-AdaptiveDistEst-SICON-2012}]
\label{lm:conn} Let $\{\mathbf{w}_{t}\}$ be an $\mathbb{R}^{NM}$-valued $\{\mathcal{F}_{t}\}$ adapted process such that $\mathbf{w}_{t}\in\mathcal{C}^{\perp}$ for all $t$,
where $\PC$ denotes the orthogonal complement of the consensus subspace $\C$, see Definition~\ref{def:consspace}. Also, let $\{L_{t}\}$ be an $\{\mathcal{F}_{t}\} $-adapted sequence of Laplacians satisfying Assumption~\ref{ass:conn}. Then there exists an $\{\mathcal{F}_{t+1}\}$ adapted $\mathbb{R}_{+}$-valued process $\{r_{t}\}$ (depending on $\{\mathbf{w}_{t}\}$ and $\{L_{t}\}$), a deterministic time $t_{r}$ (large enough), and a constant $c_{r}>0$, such that $0\leq r_{t}\leq 1$ a.s. and
\begin{equation}
\label{lm:conn20}
\left\|\left(I_{NM}-\beta_{t}L_{t}\otimes I_{M}\right)\mathbf{w}_{t}\right\|\leq\left(1-r_{t}\right)\left\|\mathbf{w}_{t}\right\|
\end{equation}
with
\begin{equation}
\label{lm:conn2}
\mathbb{E}\left[r_{t}~|~\mathcal{F}_{t}\right]\geq\frac{c_{r}}{(t+1)^{\tau_{2}}}~~\mbox{a.s.}
\end{equation}
for all $t\geq t_{r}$, where the weight sequence $\{\beta_{t}\}$ and $\tau_{2}$ are defined in~\eqref{weight}.\\
\end{lemma}

\begin{proof}[Proof of Lemma~\ref{lm:consrate}] Recall the $\{\mathcal{F}_{t}\}$-adapted process $\{\mathbf{z}_{t}\}$ with $\mz_{t}=\vecc(\mz_{n}(t))$ for all $t$, and note that by~\eqref{th:genest1} we have
\begin{equation}
\label{lm:consrate6}
\mz_{t+1}=\left(I_{NM}-\beta_{t}L_{t}\otimes I_{M}\right)\mz_{t}-\alpha_{t}\bK_{t}\left(\bh(\mz_{t})-\bh(\abbtheta)\right)+\alpha_{t}\bK_{t}\left(\bg(\mathbf{y}_{t})-\bh(\abbtheta)\right),
\end{equation}
the functions $\bh(\cdot)$ and $\bg(\cdot)$ being defined in~\eqref{lm:bg4}. For each $n$, let $\bz_{n}(t)=\mz_{n}(t)-\mz^{a}_{t}$, and denote by $\{\bz_{t}\}$ the $\{\mathcal{F}_{t}\}$-adapted process where $\bz_{t}=\vecc(\bz_{n}(t))$ for all $t$. Using the fact $(L_{t}\otimes I_{M})(\mathbf{1}_{N}\otimes\mz^{a}_{t})=\mathbf{0}$, we have
\begin{equation}
\label{lm:consrate7}\bz_{t+1}=\left(I_{NM}-\beta_{t}L_{t}\otimes I_{M}\right)\bz_{t}-\alpha_{t}\bU^{\prime}_{t}+\alpha_{t}\bJ^{\prime}_{t},
\end{equation}
where $\{\bU^{\prime}_{t}\}$ and $\{\bJ^{\prime}_{t}\}$ are $\{\mathcal{F}_{t}\}$-adapted and $\{\mathcal{F}_{t+1}\}$-adapted processes given by
\begin{equation}
\label{lm:consrate8}\bU^{\prime}_{t}=\left(I_{NM}-(\mathbf{1}_{N}\mathbf{1}_{N}^{\top})\otimes I_{M}\right)\bK_{t}\left(\bh(\mz_{t})-\bh(\abbtheta)\right),
\end{equation}
and
\begin{equation}
\label{lm:consrate9} \bJ^{\prime}_{t} = \left(I_{NM}-(\mathbf{1}_{N}\mathbf{1}_{N}^{\top})\otimes I_{M}\right)\bK_{t}\left(\bg(\mathbf{y}_{t})-\bh(\abbtheta)\right)
\end{equation}
respectively. Note that by hypothesis $\sup_{t}\|\bK_{t}\|<\infty$ a.s., and, by Theorem~\ref{th:genest}, $\sup_{t}\|\mz_{t}\|<\infty$ a.s. Hence, by the linear growth condition on $\bh(\cdot)$ (see Assumption~\ref{ass:lingrowth}), there exists an $\{\mathcal{F}_{t}\}$-adapted process $\{U^{\prime}_{t}\}$ such that, $\|\bU^{\prime}_{t}\|\leq U^{\prime}_{t}$ for all $t$ and $\sup_{t\geq 0}\|U^{\prime}_{t}\|<\infty$ a.s. Then, defining $U_{t}$ to be
\begin{equation}
\label{lm:consrate10}U_{t}=U^{\prime}_{t}\bigvee\left\|\left(I_{NM}-(\mathbf{1}_{N}\mathbf{1}_{N}^{\top})\otimes I_{M}\right)\bK_{t}\right\|~~\forall t,
\end{equation}
we have by~\eqref{lm:consrate8}-\eqref{lm:consrate10}
\begin{equation}
\label{lm:consrate11}\|\bU^{\prime}_{t}\|+\|\bJ^{\prime}_{t}\|\leq U_{t}\left(1+J_{t}\right),
\end{equation}
with $\{U_{t}\}$ being $\{\mathcal{F}_{t}\}$-adapted and $\{J_{t}\}$ being the $\{\mathcal{F}_{t+1}\}$-adapted process, $J_{t}=\|\bg(\mathbf{y}_{t})-\bh(\abbtheta)\|$ for all $t$. Note that for every $\Vap>0$ we have
\begin{equation}
\label{lm:consrate12}\East\left[J_{t}^{2+\Vap}\right]<\infty,
\end{equation}
which follows from the fact that $\bg(\mathbf{y}_{t})$ possesses moments of all orders (see Proposition~\ref{prop:analytic}). Hence, by~\eqref{lm:consrate7} we obtain
\begin{equation}
\label{lm:consrate13}
\left\|\bz_{t+1}\right\|\leq\left\|\left(I_{NM}-\beta_{t}L_{t}\otimes I_{M}\right)\bz_{t}\right\|+\alpha_{t}U_{t}\left(1+J_{t}\right)~~\forall t.
\end{equation}
Observe that, by construction, $\bz_{t}\in\PC$ for all $t$, and hence, by Lemma~\ref{lm:conn}, there exists an $\{\mathcal{F}_{t+1}\}$ adapted $\mathbb{R}_{+}$-valued process $\{r_{t}\}$ (depending on $\{\bz_{t}\}$ and $\{L_{t}\}$), a deterministic time $t_{r}$ (large enough), and a constant $c_{r}>0$, such that $0\leq r_{t}\leq 1$ a.s. and
\begin{equation}
\label{lm:consrate14}
\left\|\left(I_{NM}-\beta_{t}L_{t}\otimes I_{M}\right)\bz_{t}\right\|\leq\left(1-r_{t}\right)\left\|\bz_{t}\right\|
\end{equation}
with
\begin{equation}
\label{lm:consrate15}
\mathbb{E}\left[r_{t}~|~\mathcal{F}_{t}\right]\geq\frac{c_{r}}{(t+1)^{\tau_{2}}}~~\mbox{a.s.}
\end{equation}
for all $t\geq t_{r}$. We then have by~\eqref{lm:consrate13}-\eqref{lm:consrate14}
\begin{equation}
\label{lm:consrate16}
\left\|\bz_{t+1}\right\|\leq\left(1-r_{t}\right)\left\|\bz_{t}\right\|+\alpha_{t}U_{t}\left(1+J_{t}\right)~~\forall t.
\end{equation}
Now consider arbitrary $\Vap>0$ and note that, under the moment condition~\eqref{lm:consrate12}, the stochastic recursion in~\eqref{lm:consrate16} falls under the purview of Lemma~\ref{lm:mean-conv} (by taking $\delta_{1}$ and $\delta_{2}$ in Lemma~\ref{lm:mean-conv} to be $\tau_{2}$ and 1 respectively), and we conclude that $(t+1)^{\tau}\|\bz_{t}\|\rightarrow 0$ as $t\rightarrow\infty$ a.s. for each $\tau\in (0, 1-\tau_{2}-1/(2+\Vap))$. Noting that
\begin{equation}
\label{lm:consrate17}
\left\|\mz_{n}(t)-\mz_{l}(t)\right\|\leq\left\|\mz_{n}(t)-\mz^{a}_{t}\right\|+\left\|\mz_{l}(t)-\mz^{a}_{t}\right\|\leq 2\left\|\bz_{t}\right\|
\end{equation}
for each pair $n$ and $l$ of agents, we may further conclude that
\begin{equation}
\label{lm:consrate18}\Past\left(\lim_{t\rightarrow\infty}(t+1)^{\tau}\|\mathbf{z}_{n}(t)-\mathbf{z}_{l}(t)\|=0\right)=1,
\end{equation}
for all $\tau\in (0,1-\tau_{2}-1/(2+\Vap))$.

Since the above holds for arbitrary $\Vap>0$, the desired assertion follows by making $\Vap$ tend to $\infty$.
\end{proof}

\begin{proof}[Proof of Corollary~\ref{corr:genrate}] Since $(t+1)^{\tau_{3}}\|K_{n}(t)-\mathcal{K}\|\rightarrow 0$ a.s. as $t\rightarrow\infty$ for all $n$, by Egorov's theorem, for each $\Vap>0$, there exist a deterministic $t_{\Vap}>0$ and a positive constant $c_{\Vap}$, such that
\begin{equation}
\label{corr:genrate1}\Past\left(\sup_{t\geq t_{\Vap}}(t+1)^{\tau_{3}}\left\|K_{n}(t)-\mathcal{K}\right\|>c_{\Vap}\right)<\Vap
\end{equation}
for all $n$. Now, for such an $\Vap>0$, define, for each $n$, the following $\{\mathcal{F}_{t}\}$-adapted sequence $\{K^{\Vap}_{n}(t)\}$:
\begin{equation}
\label{corr:genrate2}K_{n}^{\Vap}(t)=\left\{\begin{array}{ll}
                                    K_{n}(t) & \mbox{if $t<t_{\Vap}$}\\
                                    K_{n}(t) & \mbox{if $t\geq t_{\Vap}$ and $\|K_{n}(t)-\mathcal{K}\|\leq c_{\Vap}(t+1)^{-\tau_{3}}$}\\
                                    \mathcal{K} & \mbox{otherwise}.
                                    \end{array}\right.
\end{equation}
Note that, by the above construction, we have $\|K_{n}(t)-\mathcal{K}\|\leq c_{\Vap}(t+1)^{-\tau_{3}}$ for all $t\geq t_{\Vap}$; hence, choosing $\tau^{\prime}\in (0,\tau_{3})$ to be a constant (independent of $\Vap$), we have that
\begin{equation}
\label{corr:genrate3}(t+1)^{\tau^{\prime}}\left\|K^{\Vap}_{n}(t)-\mathcal{K}\right\|\leq c_{\Vap}(t+1)^{-(\tau_{3}-\tau^{\prime})}~~~\forall t\geq t_{\Vap}
\end{equation}
for all $n$ and each $\Vap>0$. Thus, clearly, for each $\Vap>0$ and all $n$, the sequence $\{K^{\Vap}_{n}(t)\}$ converges a.s. to $\mathcal{K}$ uniformly (over sample paths) at rate $\tau^{\prime}>0$, i.e., for each $\delta>0$, there exists (deterministic) $t_{\delta}>0$ such that
\begin{equation}
\label{corr:genrate4}\Past\left(\sup_{t\geq t_{\delta}}(t+1)^{\tau^{\prime}}\left\|K_{n}^{\Vap}(t)-\mathcal{K}\right\|\leq\delta\right)=1.
\end{equation}
Now, for each $\Vap>0$, let us define the $\{\mathcal{F}_{t}\}$-adapted sequences $\{\mathbf{z}^{\Vap}_{n}(t)\}$, $n=1,\cdots,N$, evolving as
\begin{equation}
\label{corr:genrate5}\mathbf{z}_{n}^{\Vap}(t+1)=\mathbf{z}_{n}^{\Vap}(t)-\beta_{t}\sum_{l\in\Omega_{n}(t)}\left(\mathbf{z}_{n}^{\Vap}(t)-\mathbf{z}_{l}^{\Vap}(t)\right)+\alpha_{t}K_{n}^{\Vap}(t)\left(g_{n}(\mathbf{y}_{n}(t))-h_{n}(\mathbf{z}_{n}^{\Vap}(t))\right).
\end{equation}
Noting that
\begin{equation}
\label{corr:genrate6}
\left\{\sup_{n,t}\left\|\mathbf{z}^{\Vap}_{n}(t)-\mathbf{z}_{n}(t)\right\|=0\right\}~~\mbox{on}~~\left\{\sup_{n,t}\left\|K_{n}^{\Vap}(t)-K_{n}(t)\right\|=0\right\},
\end{equation}
we have
\begin{equation}
\label{corr:genrate7}\Past\left(\sup_{n,t}\left\|\mathbf{z}^{\Vap}_{n}(t)-\mathbf{z}_{n}(t)\right\|=0\right)\geq 1-N\Vap
\end{equation}
by~\eqref{corr:genrate1}-\eqref{corr:genrate2}.

The uniform convergence of the gain sequences $\{K^{\Vap}_{n}(t)\}$ to $\mathcal{K}$ at rate $\tau^{\prime}>0$ ensures that, for each $\Vap>0$, the processes $\{\mathbf{z}_{n}^{\Vap}(t)\}$ satisfy the hypotheses of Theorem~\ref{th:genrate} and, hence, there exists a positive constant $\mu$ (that depends on $\tau^{\prime}$ but not $\Vap$), such that $(t+1)^{\mu}\|\mathbf{z}_{n}^{\Vap}(t)-\abtheta\|\rightarrow 0$ as $t\rightarrow 0$ a.s. for each $n$. Hence, by~\eqref{corr:genrate7} we have
\begin{equation}
\label{corr:genrate8}\Past\left(\lim_{t\rightarrow\infty}(t+1)^{\mu}\left\|\mathbf{z}_{n}(t)-\abtheta\right\|=0\right)\geq 1-N\Vap
\end{equation}
for all $n$. Since $\Vap>0$ is arbitrary and $\mu$ does not depend on $\Vap$, we may further conclude from~\eqref{corr:genrate8} that
$(t+1)^{\mu}\|\mathbf{z}_{n}(t)-\abtheta\|\rightarrow 0$ as $t\rightarrow\infty$ a.s. for all $n$.
\end{proof}

\section{Proofs in Section~\ref{sec:proof_main_res}}
\label{sec:app2}

{\bf Proof of Lemma~\ref{lm:gainconv}}: The proof of Lemma~\ref{lm:gainconv} is accomplished in two steps: first, we show that the gain sequences reach consensus, and subsequently demonstrate that the limiting consensus value is indeed $I^{-1}(\abtheta)$. To this end, consider the following:
\begin{lemma}
\label{lm:gaincons} Recall for each $n$, the $\{\mathcal{F}_{t}\}$-adapted sequence $\{G_{n}(t)\}$ evolving as in~\eqref{gain2}, and denote by $\{G^{a}_{t}\}$ their instantaneous network averages, i.e., $G^{a}_{t}=(1/N)\sum_{n=1}^{N}G^{a}_{n}(t)$ for all $t$. Then, for each $n$ and $\tau\in [0,1-\tau_{2})$, we have
\begin{equation}
\label{lm:gaincons1}\Past\left(\lim_{t\rightarrow\infty}(t+1)^{\tau}\left\|G_{n}(t)-G^{a}(t)\right\|=0\right)=1,
\end{equation}
where $\tau_{2}$ is the exponent associated with the weight sequence $\{\beta_{t}\}$, see Assumption~\ref{ass:weight}.
\end{lemma}

\begin{proof} We will show the desired convergence in the matrix Frobenius norm (denoted by $\|\cdot\|_{F}$ in the following), the convergence in the induced $\mathcal{L}_{2}$ sense following immediately.  Note that, by Lemma~\ref{lm:auxcons}, $\bx_{n}(t)\rightarrow\abtheta$ as $t\rightarrow\infty$ a.s. for all $n$, hence, for each $n$, by the continuity of the local Fisher information matrix $I_{n}(\cdot)$, we have that $I_{n}(\bx_{n}(t))\rightarrow I_{n}(\abtheta)$ as $\tri$ a.s.

Let $\bG_{n}(t)=G_{n}(t)-G^{a}_{t}$ denote the deviation at agent $n$ from the instantaneous network average $G^{a}_{t}$ and $I^{a}_{t}=(1/N)\sum_{n=1}^{N}I_{n}(\bx_{n}(t))$ the network average of the $I_{n}(\bx_{n}(t))$'s. Also, let $\bG_{t}$ and $\bI_{t}$ denote the matrices $\vecc\left(\bG_{n}(t)\right)$ and $\vecc\left(\bI_{n}(t)\right)$ respectively, where $\bI_{n}(t)=I_{n}(\bx_{n}(t))-I^{a}_{t}$ for all $n$. Using the following readily verifiable properties of the Laplacian $L_{t}$
\begin{equation}
\label{lm:gaincons3}\left(\mathbf{1}_{N}\otimes I_{M}\right)^{\top}\left(L_{t}\otimes I_{M}\right)=\mathbf{0}~~\mbox{and}~~\left(L_{t}\otimes I_{M}\right)\left(\mathbf{1}_{N}\otimes G^{a}_{t}\right)=\mathbf{0},
\end{equation}
we have by~\eqref{gain2}
\begin{equation}
\label{lm:gaincons4}\bG_{t+1}=\left(I_{NM}-\beta_{t}\left(L_{t}\otimes I_{M}\right)-\alpha_{t}I_{NM}\right)\bG_{t}+\alpha_{t}\bI_{t}
\end{equation}
for all $t\geq 0$.

Since for all $n$, $I_{n}(\bx_{n}(t))\rightarrow I_{n}(\abtheta)$ as $\tri$ a.s., the sequences $\{I_{n}(\bx_{n}(t))\}$ are bounded a.s. and, in particular, there exists an $\{\mathcal{F}_{t}\}$-adapted a.s. bounded process $\{U_{t}\}$ such that $\|\bI_{t}\|_{F}\leq U_{t}$ for all $t$. For $m\in\{1,\cdots,M\}$, denote by $\bG_{m,t}$ the $m$-th column of $\bG_{t}$. Clearly, the process $\{\bG_{m,t}\}$ is $\{\mathcal{F}_{t}\}$-adapted and $\bG_{m,t}\in\PC$ for all $t$. Hence, by Lemma~\ref{lm:conn}, there exist a $[0,1]$-valued $\{\mathcal{F}_{t+1}\}$-adapted process $\{r_{m,t}\}$ and a positive constant $c_{m,r}$ such that
\begin{equation}
\label{lm:gaincons5}\left\|\left(I_{NM}-\beta_{t}L_{t}\otimes I_{M}\right)\bG_{m,t}\right\|\leq \left(1-r_{m,t}\right)\left\|\bG_{m,t}\right\|
\end{equation}
and $\East[r_{m,t}|\mathcal{F}_{t}]\geq c_{m,r}/(t+1)^{\tau_{2}}$ a.s. for all $t\geq t_{0}$ sufficiently large. Noting that the square of the Frobenius norm is the sum of the squared column $\mathcal{L}_{2}$ norms, we have
\begin{equation}
\label{lm:gaincons6}\left\|\left(I_{NM}-\beta_{t}L_{t}\otimes I_{M}\right)\bG_{t}\right\|^{2}_{F}\leq\sum_{m=1}^{M}\left(1-r_{m,t}\right)^{2}\left\|\bG_{m,t}\right\|^{2}\leq\left(1-r_{t}\right)^{2}\left\|\bG_{t}\right\|^{2}_{F},
\end{equation}
where $\{r_{t}\}$ is the $\{\mathcal{F}_{t}\}$-adapted process given by $r_{t}=r_{1,t}\wedge\cdots\wedge r_{M,t}$ for all $t$. By the conditional Jensen's inequality we obtain
\begin{equation}
\label{lm:gaincons7}\East\left[r_{t}~|~\mathcal{F}_{t}\right]\geq\bigwedge_{m=1}^{M}\East\left[r_{m,t}~|~\mathcal{F}_{t}\right]\geq c_{r}/(t+1)^{\tau_{2}}
\end{equation}
for some constant $c_{r}>0$ and all $t\geq t_{0}$. Since $\beta_{t}/\alpha_{t}\rightarrow\infty$ as $\tri$, by making $t_{0}$ larger if necessary, we obtain from~\eqref{lm:gaincons6}
\begin{align}
\label{lm:gaincons8}\left\|\left(I_{NM}-\beta_{t}L_{t}\otimes I_{M}-\alpha_{t}I_{NM}\right)\bG_{t}\right\|_{F}\leq\left\|\left(I_{NM}-\beta_{t}L_{t}\otimes I_{M}\right)\bG_{t}\right\|_{F}+\alpha_{t}\left\|\bG_{t}\right\|_{F}\\ \leq \left(1-r_{t}\right)\left\|\bG_{t}\right\|_{F}+\alpha_{t}\left\|\bG_{t}\right\|_{F}\leq\left(1-r_{t}/2\right)\left\|\bG_{t}\right\|_{F}
\end{align}
for all $t\geq t_{0}$. It then follows from~\eqref{lm:gaincons4} and~\eqref{lm:gaincons8} that
\begin{equation}
\label{lm:gaincons9}\left\|\bG_{t+1}\right\|_{F}\leq\left\|\left(I_{NM}-\beta_{t}L_{t}\otimes I_{M}-\alpha_{t}I_{NM}\right)\bG_{t}\right\|_{F}+\alpha_{t}U_{t}\leq \left(1-r_{t}/2\right)\left\|\bG_{t}\right\|_{F}+\alpha_{t}U_{t}
\end{equation}
for all $t\geq t_{0}$. Clearly, the above recursion falls under the purview of Lemma~\ref{lm:mean-conv} (by setting  $\delta_{1}$, $\delta_{2}$ and $J_{t}$ in Lemma~\ref{lm:mean-conv} to $\tau_{2}$, 1 and 0 respectively), and we conclude that $(t+1)^{\tau}\|\bG_{t}\|_{F}\rightarrow 0$ as $\tri$ a.s. for each $\tau\in [0,1-\tau_{2})$. The assertion in Lemma~\ref{lm:gaincons} follows immediately.
\end{proof}

We state another approximation result from~\cite{Fabian-1} regarding deterministic recursions to be used in the sequel.
\begin{proposition}[Lemma 4.3 in~\cite{Fabian-1}]
\label{prop:Fab-1}Let $\{b_{t}\}$ be a scalar sequence satisfying
\begin{equation}
\label{prop:Fab-11}
b_{t+1}\leq\left(1-\frac{c}{t+1}\right)b_{t}+d_{t}(t+1)^{-\tau}
\end{equation}
where $c>\tau$, $\tau>0$, and the sequence $\{d_{t}\}$ is summable. Then $\limsup_{t\rightarrow\infty}(t+1)^{\tau}b_{t}<\infty$.
\end{proposition}

We now complete the proof of Lemma~\ref{lm:gainconv}.
\begin{proof}[Proof of Lemma~\ref{lm:gainconv}]
Following the notation in the proof of Lemma~\ref{lm:gaincons} and using properties of the graph Laplacian~\eqref{lm:gaincons3}, the process $\{G^{a}_{t}\}$ (the instantaneous network average of the $G_{n}(t)$'s) may be shown to satisfy the following recursion for all $t$:
\begin{equation}
\label{lm:gainconv3}G^{a}_{t+1}=\left(1-\alpha_{t}\right)G^{a}_{t}+\alpha_{t}I^{a}_{t}.
\end{equation}
Noting that the local Fisher information matrices $I_{n}(\cdot)$ are locally Lipschitz in the argument and the fact that $\bx_{n}(t)\rightarrow\abtheta$ as $\tri$ a.s. (see Lemma~\ref{lm:auxcons}), we have that
\begin{equation}
\label{lm:gainconv4}
\left\|I^{a}_{t}-(1/N)I(\abtheta)\right\|=O\left(\vee_{n=1}^{N}\|\bx_{n}(t)-\abtheta\|\right).
\end{equation}
Since, by Lemma~\ref{lm:auxcons}, $(t+1)^{\mu_{0}}\|\bx_{n}(t)-\abtheta\|\rightarrow 0$ as $\tri$ a.s., we may further conclude from~\eqref{lm:gainconv4} that
\begin{equation}
\label{lm:gainconv5}\left\|I^{a}_{t}-(1/N)I(\abtheta)\right\|=o\left((t+1)^{-\mu_{0}}\right).
\end{equation}
Now let $\tau_{5}$ be a positive constant such that $\tau_{5}<(\mu_{0}\wedge 1)$. Noting that $\alpha_{t}=(t+1)^{-1}$ by definition, by~\eqref{lm:gainconv5} we may then conclude that there exists an $\mathbb{R}_{+}$-valued $\{\mathcal{F}_{t}\}$-adapted stochastic process $\{\wid_{t}\}$, such that,
\begin{equation}
\label{lm:gainconv6}\alpha_{t}\left\|I^{a}_{t}-(1/N)I(\abtheta)\right\|\leq \wid_{t}(t+1)^{-\tau_{5}}
\end{equation}
for all $t$, with $\{\wid_{t}\}$ satisfying
\begin{equation}
\label{lm:gainconv7}\wid_{t}=o\left((t+1)^{-1-\mu_{0}+\tau_{5}}\right).
\end{equation}
By~\eqref{lm:gainconv3} and~\eqref{lm:gainconv6} we then obtain
\begin{align}
\label{lm:gainconv8}
\left\|G^{a}_{t+1}-(1/N)I(\abtheta)\right\| & \leq\left(1-(t+1)^{-1}\right)\left\|G^{a}_{t}-(1/N)I(\abtheta)\right\|\\ & +\alpha_{t}\left\|I^{a}_{t}-(1/N)I(\abtheta)\right\|\\ &
\leq \left(1-(t+1)^{-1}\right)\left\|G^{a}_{t}-(1/N)I(\abtheta)\right\|+\wid_{t}(t+1)^{-\tau_{5}}
\end{align}
for all $t$. Further, by~\eqref{lm:gainconv7}, we have $\sum_{t}\wid_{t}<\infty$ a.s. (since $\tau_{5}<\mu_{0}$ by construction); also noting that $\tau_{5}<1$ (again by construction), a pathwise application of Proposition~\ref{prop:Fab-1} to the stochastic recursion~\eqref{lm:gainconv8} yields
\begin{equation}
\label{lm:gainconv9}\limsup_{\tri}(t+1)^{\tau_{5}}\left\|G^{a}_{t}-(1/N)I(\abtheta)\right\|<\infty~~\mbox{a.s.},
\end{equation}
from which we may further conclude that $(t+1)^{\tau_{6}}\left\|G^{a}_{t}-(1/N)I(\abtheta)\right\|\rightarrow 0$ as $\tri$ a.s., where $\tau_{6}$ is another positive constant such that $\tau_{6}<\tau_{5}$.

Now introducing another constant $\tau_{7}$ such that $0<\tau_{7}<(1-\tau_{2})\wedge\tau_{6}$, by Lemma~\ref{lm:gaincons} it may be readily concluded that
\begin{equation}
\label{lm:gainconv10}
\Past\left(\lim_{\tri}(t+1)^{\tau_{7}}\left\|G_{n}(t)-(1/N)I(\abtheta)\right\|=0\right)=1
\end{equation}
for all $n$. Finally, noting that matrix inversion is a locally Lipschitz operator in a neighborhood of an invertible argument, we have by~\eqref{gain1},~\eqref{ass:weight3}, and~\eqref{lm:gainconv10} that
\begin{align}
\label{lm:gainconv11}
\left\|K_{n}(t)-N.I^{-1}(\abtheta)\right\|&=\left\|\left(G_{n}(t)+\varphi_{t}I_{M}\right)^{-1}-N.I^{-1}(\abtheta)\right\|\\ &=O\left(\left\|G_{n}(t)-(1/N)I(\abtheta)\right\|+\varphi_{t}\right) =O\left((t+1)^{-\tau_{7}}+(t+1)^{-\mu_{2}}\right)\\ & =o\left((t+1)^{-\tau^{\prime}}\right),
\end{align}
where $\tau^{\prime}$ may be taken to be an arbitrary positive constant satisfying $\tau^{\prime}<\tau_{7}\wedge\mu_{2}$. Hence, the desired assertion follows.
\end{proof}

{\bf Proof of Lemma~\ref{lm:dev}}: The following intermediate approximation will be used in the proof of Lemma~\ref{lm:dev}.
\begin{lemma}
\label{lm:devcons} For each $n$, let $\{\mathbf{x}_{n}(t)\}$ and $\{\bv_{n}(t)\}$ be as in the hypothesis of Lemma~\ref{lm:dev} and denote by $\{\bu_{n}(t)\}$ the $\{\mathcal{F}_{t}\}$-adapted process such that $\bu_{n}(t)=\mathbf{x}_{n}(t)-\bv_{n}(t)$ for all $t$. Then, for each $\gamma\in [0,1-\tau_{2})$ \textup{(}where $\tau_{2}$ is the exponent corresponding to $\{\beta_{t}\}$, see Assumption~\ref{ass:weight}\textup{)}, we have
\begin{equation}
\label{lm:devcons1}\Past\left(\lim_{\tri}(t+1)^{\gamma}\left\|\bu_{n}(t)-\bu_{l}(t)\right\|=0\right)=1
\end{equation}
for all pairs $(n,l)$ of network agents.
\end{lemma}

\begin{proof} By~\eqref{opt:1} and~\eqref{corr:linz1}, the process $\{\bu_{n}(t)\}$ is readily seen to satisfy the recursions
\begin{equation}
\label{lm:devcons2}
\bu_{n}(t+1)=\bu_{n}(t)-\beta_{t}\sum_{n=1}^{N}\left(\bu_{n}(t)-\bu_{l}(t)\right)-\alpha_{t}K_{n}(t)\mathbf{U}^{\prime}_{n}(t),
\end{equation}
where
\begin{equation}
\label{lm:devcons3}\mathbf{U}_{n}^{\prime}(t)=h_{n}(\mathbf{x}_{n}(t))-h_{n}(\abtheta)-I_{n}(\abtheta)\left(\bv_{n}(t)-\abtheta\right)
\end{equation}
for all $t$. Noting that the processes $\{\mathbf{x}_{n}(t)\}$ and $\{\bv_{n}(t)\}$ converge a.s. as $\tri$ (see Corollary~\ref{corr:xconv} and Corollary~\ref{corr:linz}), we conclude that the sequence $\{\mpu_{n}(t)\}$, thus defined, is bounded a.s. Denoting by $\mpu_{t}$ and $\bu_{t}$ the block-vectors $\vecc\left(\mpu_{n}(t)\right)$ and $\vecc\left(\bu_{n}(t)\right)$ respectively, from~\eqref{lm:devcons2} we then have
\begin{equation}
\label{lm:devcons4}
\bu_{t+1}=\left(I_{NM}-\beta_{t}L_{t}\otimes I_{M}\right)\bu_{t}-\alpha_{t}\mathbf{u}_{t},
\end{equation}
where $\{\mathbf{u}_{t}\}$ is the $\{\mathcal{F}_{t}\}$-adapted process given by $\mathbf{u}_{t}=\ndiag\left(K_{n}(t)\right).\mpu_{t}$ for all $t$. Noting that $\{\mpu_{t}\}$ is bounded a.s. and the adaptive gain sequence $\{K_{n}(t)\}$ converges a.s. as $\tri$ for all $n$, the process $\{\mathbf{u}_{t}\}$ is readily seen to be bounded a.s. Further, denoting by $\{\wu_{t}\}$ and $\{\wU_{t}\}$ the processes, such that,
\begin{equation}
\label{lm:devcons5}
\wu_{t}=\left(I_{NM}-\mathbf{1}_{N}.\left(\mathbf{1}_{N}\otimes I_{M}\right)^{\top}\right)\bu_{t}~~\mbox{and}~~\wU_{t}=\left(I_{NM}-\mathbf{1}_{N}.\left(\mathbf{1}_{N}\otimes I_{M}\right)^{\top}\right)\mathbf{u}_{t}
\end{equation}
for all $t$, we have (using standard properties of the Laplacian)
\begin{equation}
\label{lm:devcons6}
\wu_{t}=\left(I_{NM}-\beta_{t}L_{t}\otimes I_{M}\right)\wu_{t}-\alpha_{t}\wU_{t}
\end{equation}
for all $t$. Clearly, $\wu_{t}\in\PC$ for all $t$, and we may note that, at this point the evolution~\eqref{lm:devcons6} resembles the dynamics analyzed in Lemma~\ref{lm:consrate} (for the process $\{\bz_{t}\}$, see~\eqref{lm:consrate13}). Following essentially similar arguments as in~\eqref{lm:consrate13}-\eqref{lm:consrate18}, we have $(t+1)^{\gamma}\|\wu_{t}\|\rightarrow 0$ as $\tri$ a.s. for all $\gamma\in [0,1-\tau_{2})$, from which the desired assertion follows.
\end{proof}

\begin{proof}[Proof of Lemma~\ref{lm:dev}] In what follows we stick to the notation in the proof of Lemma~\ref{lm:devcons}. By~\eqref{lm:devcons2} we have that
\begin{equation}
\label{lm:dev8}\bau_{t+1}=\bau_{t}-(1/N)\alpha_{t}\sum_{n=1}^{N}K_{n}(t)\mpu_{n}(t),
\end{equation}
where $\bau_{t}=(1/N)\sum_{n=1}^{N}\bu_{n}(t)$ for all $t$. Now note that, for each $n$, the function $h_{n}(\cdot)$ is twice continuously differentiable with gradient $I_{n}(\cdot)$ (see Proposition~\ref{prop:analytic} and Proposition~\ref{prop:inf}), and hence there exist positive constants $\oc$ and $R$, such that for each $n$,
\begin{equation}
\label{lm:dev9}\left\|h_{n}(\mathbf{z})-h_{n}(\abtheta)-I_{n}(\abtheta)\left(\mathbf{z}-\abtheta\right)\right\|\leq\oc\left\|\mathbf{z}-\abtheta\right\|^{2}
\end{equation}
for all $\mathbf{z}\in\mathbb{R}^{M}$ with $\|\mathbf{z}-\abtheta\|\leq R$. Since $\mathbf{x}_{n}(t)\rightarrow\abtheta$ as $\tri$ a.s. (see Corollary~\ref{corr:xconv}) for each $n$, there exists a finite random time $t_{R}$ such that
\begin{equation}
\label{lm:dev10}\max_{n=1}^{N}\left\|\mathbf{x}_{n}(t)-\abtheta\right\|\leq R~~\mbox{$\forall t\geq t_{R}$ a.s.}
\end{equation}
Hence, by~\eqref{lm:devcons3} and~\eqref{lm:dev9}-\eqref{lm:dev10}, we have that
\begin{align}
\label{lm:dev11}
\mpu_{n}(t) &= h_{n}(\mathbf{x}_{n}(t))-h_{n}(\abtheta)-I_{n}(\abtheta)\left(\bv_{n}(t)-\abtheta\right)
=I_{n}(\abtheta)\left(\mathbf{x}_{n}(t)-\bv_{n}(t)\right)+\mR_{n}(t)\\
&=I_{n}(\abtheta)\bu_{n}(t)+\mR_{n}(t),
\end{align}
for all $n$ and $t$, where the residuals $\mR_{n}(t)$, $n=1,\cdots,N$ satisfy
\begin{equation}
\label{lm:dev12}
\left\|\mR_{n}(t)\right\|\leq\oc\left\|\mathbf{x}_{n}(t)-\abtheta\right\|^{2}~~\mbox{$\forall t\geq t_{R}$ a.s.}
\end{equation}
Standard algebraic manipulations further yield
\begin{align}
\label{lm:dev14}\left\|\mR_{n}(t)\right\| & \leq\oc\left\|\mathbf{x}_{n}(t)-\abtheta\right\|^{2}
 \leq 2\oc\left\|\mathbf{x}_{n}(t)-\bv_{n}(t)\right\|^{2}+2\oc\left\|\bv_{n}(t)-\abtheta\right\|^{2}\\ & =2\oc\left\|\bu_{n}(t)\right\|^{2}+2\oc\left\|\bv_{n}(t)-\abtheta\right\|^{2}  \leq 4\oc\left\|\bau_{t}\right\|^{2}+4\oc\left\|\bu_{n}(t)-\bau_{t}\right\|^{2}+2\oc\left\|\bv_{n}(t)-\abtheta\right\|^{2}
\end{align}
for all $t\geq t_{R}$ a.s.

Note that, the fact that $(t+1)^{\tau}\|\bv_{n}(t)-\abtheta\|\rightarrow 0$ as $\tri$ a.s. for all $n$ and $\tau\in [0,1/2)$ implies that there exists a constant $\gamma_{1}>1/2$ such that
\begin{equation}
\label{lm:dev13}\max_{n=1}^{N}\left\|\bv_{n}(t)-\abtheta\right\|^{2}=o\left((t+1)^{-\gamma_{1}}\right)~~\mbox{a.s.}.
\end{equation}

Also, by Lemma~\ref{lm:devcons} and the fact that $\tau_{2}<1/2$ (see Assumption~\ref{ass:weight}), we have that
\begin{equation}
\label{lm:dev15}
\max_{n=1}^{N}\left\|\bu_{n}(t)-\bau_{t}\right\|=o\left((t+1)^{-\gamma_{2}}\right)~~\mbox{a.s.}
\end{equation}
for some constant $\gamma_{2}>1/2$.

By the previous construction, the recursions for $\{\bau_{t}\}$ may be written as
\begin{equation}
\label{lm:dev16}
\bau_{t+1}=\bau_{t}-\alpha_{t}Q_{t}\bau_{t}-\alpha_{t}\wmR_{t},
\end{equation}
where
\begin{equation}
\label{lm:dev17}
Q_{t}=(1/N)\sum_{n=1}^{N}K_{n}(t)I_{n}(\abtheta),
\end{equation}
and
\begin{equation}
\label{lm:dev18}
\wmR_{t}=(1/N)\sum_{n=1}^{N}K_{n}(t)\left(I_{n}(\abtheta)\left(\bu_{n}(t)-\bau_{t}\right)+\mR_{n}(t)\right)
\end{equation}
for all $t$. By~\eqref{lm:dev14} and \eqref{lm:dev18} we obtain
\begin{align}
\label{lm:dev19}\left\|\wmR_{t}\right\| & \leq (1/N)\sum_{n=1}^{N}\left\|K_{n}(t)I_{n}(\abtheta)\right\|\left\|\bu_{n}(t)-\bau_{t}\right\|\\ & +(4\oc/N)\sum_{n=1}^{N}\left\|K_{n}(t)I_{n}(\abtheta)\right\|\left\|\bau_{t}\right\|^{2}
+(4\oc/N)\sum_{n=1}^{N}\left\|K_{n}(t)I_{n}(\abtheta)\right\|\left\|\bu_{n}(t)-\bau_{t}\right\|^{2}\\ & +(2\oc/N)\sum_{n=1}^{N}\left\|K_{n}(t)I_{n}(\abtheta)\right\|\left\|\bv_{n}(t)-\abtheta\right\|^{2}
\end{align}
for $t\geq t_{R}$ a.s. Then, denoting by $\{\blambda_{t}\}$ the $\{\mathcal{F}_{t}\}$-adapted process such that
\begin{equation}
\label{lm:dev20}\blambda_{t}=(4\oc/N)\sum_{n=1}^{N}\left\|K_{n}(t)I_{n}(\abtheta)\right\|\left\|\bau_{t}\right\|
\end{equation}
for all $t$, and observing that, by~\eqref{lm:dev13} and~\eqref{lm:dev15},
\begin{equation}
\label{lm:dev21}
(2\oc/N)\sum_{n=1}^{N}\left\|K_{n}(t)I_{n}(\abtheta)\right\|\left\|\bv_{n}(t)-\abtheta\right\|^{2}=o\left((t+1)^{-\gamma_{1}}\right)~~\mbox{a.s.},
\end{equation}
\begin{equation}
\label{lm:dev22}
(1/N)\sum_{n=1}^{N}\left\|K_{n}(t)I_{n}(\abtheta)\right\|\left\|\bu_{n}(t)-\bau_{t}\right\|=o\left((t+1)^{-\gamma_{2}}\right)~~\mbox{a.s.},
\end{equation}
and
\begin{equation}
\label{lm:dev200}
(4\oc/N)\sum_{n=1}^{N}\left\|K_{n}(t)I_{n}(\abtheta)\right\|\left\|\bu_{n}(t)-\bau_{t}\right\|^{2}=o\left((t+1)^{-\gamma_{2}}\right)~~\mbox{a.s.}
\end{equation}
(note that the gain sequences $\{K_{n}(t)\}$'s converge a.s., hence, $K_{n}(t)=O(1)$ for all $n$), we obtain the following from~\eqref{lm:dev19}:
\begin{equation}
\label{lm:dev23}\left\|\wmR_{t}\right\|\leq \blambda_{t}\left\|\bau_{t}\right\|+o\left((t+1)^{-\gamma_{3}}\right)
\end{equation}
for some constant $\gamma_{3}$ such that $1/2<\gamma_{3}<\gamma_{1}\wedge\gamma_{2}$. Thus, by~\eqref{lm:dev17}-\eqref{lm:dev18} and~\eqref{lm:dev23}, and by making $t_{R}$ larger if necessary, we conclude that there exists a positive constant $b$ such that
\begin{equation}
\label{lm:dev24}
\left\|\bau_{t+1}\right\|\leq\left\|I_{M}-\alpha_{t}Q_{t}+\alpha_{t}\blambda_{t}I_{M}\right\|\left\|\bau_{t}\right\|+b\alpha_{t}(t+1)^{-\gamma_{3}}
\end{equation}
for all $t\geq t_{R}$ a.s.
Since $K_{n}(t)\rightarrow N.I^{-1}(\abtheta)$ as $\tri$ a.s. for each $n$ and $\sum_{n=1}^{N}I_{n}(\abtheta)=I(\abtheta)$, we have $Q_{t}\rightarrow I_{M}$ as $\tri$ a.s.; similarly, since for all $n$ the sequences $\{\mathbf{x}_{n}(t)\}$ and $\{\bv_{n}(t)\}$ converge to $\abtheta$ a.s. as $\tri$ (see Corollary~\ref{corr:xconv} and Corollary~\ref{corr:linz}), it follows (from definition) that $\bu_{n}(t)\rightarrow 0$ as $\tri$ a.s. for all $n$, and hence $\blambda_{t}\rightarrow 0$ as $\tri$ a.s. The fact that, $Q_{t}\rightarrow I_{M}$ and $\blambda_{t}\rightarrow 0$ as $\tri$ a.s., ensures that, by making $t_{R}$ larger if necessary, the following holds
\begin{equation}
\label{lm:dev25}\left\|I_{M}-\alpha_{t}Q_{t}+\alpha_{t}\blambda_{t}I_{M}\right\|\leq 1-(2/3).\alpha_{t}=1-(2/3).(t+1)^{-1}
\end{equation}
for all $t\geq t_{R}$ a.s. Let $\gamma_{4}$ be a constant such that $1/2<\gamma_{4}<\gamma_{3}\wedge (2/3)$; then, by~\eqref{lm:dev24}-\eqref{lm:dev25}, we have
\begin{equation}
\label{lm:dev26}
\left\|\bau_{t+1}\right\|\leq\left(1-(2/3).(t+1)^{-1}\right)\left\|\bau_{t}\right\|+d_{t}(t+1)^{-\gamma_{4}}
\end{equation}
for all $t\geq t_{R}$ a.s., where $d_{t}=b\alpha_{t}(t+1)^{\gamma_{4}-\gamma_{3}}$. Since $\gamma_{4}<2/3$ and the sequence $\{d_{t}\}$ is summable, a pathwise application of Proposition~\ref{prop:Fab-1} yields
\begin{equation}
\label{lm:dev27}
\Past\left(\limsup_{\tri}(t+1)^{\gamma_{4}}\left\|\bau_{t}\right\|<\infty\right)=1.
\end{equation}
Hence, by choosing $\overline{\tau}$, such that $1/2<\overline{\tau}<\gamma_{2}\wedge\gamma_{4}$ (where $\gamma_{2}$ is defined in~\eqref{lm:dev15}), we have that $(t+1)^{\overline{\tau}}\bu_{n}(t)\rightarrow 0$ as $\tri$ a.s. for all $n$ and the desired assertion follows.
\end{proof}

\bibliographystyle{IEEEtran}
\bibliography{IEEEabrv,CentralBib}

\end{document}